\documentclass{amsart}
\usepackage[dvipsnames]{xcolor}
\usepackage{amssymb}
\usepackage[utf8]{inputenc}
\usepackage[british]{babel}
\usepackage{tikz}
\usepackage{tikz-cd}
\usepackage{amsmath, amsthm, amssymb, amsfonts}
\usepackage{hyperref}
\usepackage{mathtools}
\usepackage[retainorgcmds]{IEEEtrantools}
\usepackage{latexsym}
\usepackage{enumerate}
\usepackage{nomencl}
\usepackage{nameref}
\usepackage{chngcntr}
\usepackage{bbm}
\usepackage[margin=1.6in]{geometry}

\colorlet{mylinkcolor}{Blue}
\colorlet{mycitecolor}{purple}
\colorlet{myurlcolor}{Emerald}

\hypersetup{
  linkcolor  = mylinkcolor!,
  citecolor  = mycitecolor!,
  urlcolor   = myurlcolor!,
  colorlinks = true,
}

\usetikzlibrary{fit, patterns}

\usetikzlibrary{decorations.pathmorphing}

\tikzset{%
    symbol/.style={%
        draw=none,
        every to/.append style={%
            edge node={node [sloped, allow upside down, auto=false]{$#1$}}}
    }
}

\numberwithin{equation}{section}

\newtheorem{thm}{Theorem}[section]
\newtheorem{lem}[thm]{Lemma}
\newtheorem{prop}[thm]{Proposition}
\newtheorem{cor}[thm]{Corollary}
\newtheorem{innercustomthm}{Theorem}
\newenvironment{customthm}[1]
  {\renewcommand\theinnercustomthm{#1}\innercustomthm}
  {\endinnercustomthm}

\theoremstyle{definition}
\newtheorem{defi}[thm]{Definition}
\newtheorem{ex}[thm]{Example}

\theoremstyle{remark}
\newtheorem{rem}[thm]{Remark}

\newcommand{\Rest}{\operatorname{Res}}
\newcommand{\op}{\normalfont{^{op}}}

\newcommand{\modu}{\operatorname{mod}}
\newcommand{\Mod}[1]{#1\operatorname{-mod}}

\newcommand{\MOD}[1]{#1\operatorname{-Mod}}

\newcommand{\Hom}[3]{\operatorname{Hom}_{#1}\left( #2,#3\right)  }
\newcommand{\HOM}{\operatorname{Hom}}
\newcommand{\End}[2]{\operatorname{End}_{#1}\left(#2\right) }

\newcommand{\Rad}[1]{\operatorname{Rad} #1}
\newcommand{\RAD}[2]{\operatorname{Rad}^{#1} #2}

\newcommand{\Ker}[1]{\operatorname{Ker}#1}

\newcommand{\Ext}[4]{\operatorname{Ext}_{#1}^{#2}\left(#3,#4 \right)}

\newcommand{\Z}{\mathbb{Z}}
\newcommand{\N}{\mathbb{N}}
\newcommand{\Znn}{\Z_{\geq 0}}

\newcommand{\B}[1]{\left(#1\right)}

\newcommand*{\isoarrow}[1]{\arrow[#1,"\rotatebox{90}{\(\sim\)}"]}

\tikzset{%
    symbol/.style={%
        draw=none,
        every to/.append style={%
            edge node={node [sloped, allow upside down, auto=false]{$#1$}}}
    }
}

\newcommand*\circled[1]{\tikz[baseline=(char.base)]{
            \node[shape=rectangle,draw,inner sep=1pt] (char) {#1};}}

\makeatletter
\@namedef{subjclassname@2020}{%
  \textup{2020} Mathematics Subject Classification}
\makeatother

\begin{document}

\title{All quasihereditary algebras with a regular exact Borel subalgebra}
\author{Teresa Conde}
\address{Institute of Algebra and Number Theory, University of Stuttgart\\ Pfaffenwaldring 57, 70569 Stuttgart, Germany}
\email{\href{mailto:tconde@mathematik.uni-stuttgart.de}{\nolinkurl{tconde@mathematik.uni-stuttgart.de}}}
\date{\today}

\begin{abstract}
Not every quasihereditary algebra $(A,\Phi,\unlhd)$ has an exact Borel subalgebra. A theorem by Koenig, K\"ulshammer and Ovsienko asserts that there always exists a quasihereditary algebra Morita equivalent to $A$ that has a regular exact Borel subalgebra, but a characterisation of such a Morita representative is not directly obtainable from their work. This paper gives a criterion to decide whether a quasihereditary algebra contains a regular exact Borel subalgebra and provides a method to compute all the representatives of $A$ that have a regular exact Borel subalgebra. It is shown that the Cartan matrix of a regular exact Borel subalgebra of a quasihereditary algebra $(A,\Phi,\unlhd)$ only depends on the composition factors of the standard and costandard $A$-modules and on the dimension of the $\HOM$-spaces between standard $A$-modules. We also characterise the basic quasihereditary algebras that admit a regular exact Borel subalgebra.
\end{abstract}

\maketitle
\tableofcontents

\section{Introduction}
The existence of Borel subalgebras is a key feature of semisimple Lie algebras. When $\mathfrak{b}$ is a Borel subalgebra of a complex semisimple Lie algebra $\mathfrak{g}$, the Poincar\'e--Birkhoff--Witt Theorem implies that the universal enveloping algebra $U(\mathfrak{b})$ of $\mathfrak{b}$ is a subalgebra of $U(\mathfrak{g})$ and, further, $U(\mathfrak{g})$ is free as a module over $U(\mathfrak{b})$ (\cite[§17.3, Corollary D]{humphreys2012introduction}). As a result, the induction functor $U(\mathfrak{g})\otimes_{U(\mathfrak{b})}-$, which turns $\mathfrak{b}$-representations into $\mathfrak{g}$-representations, is exact.  Additionally, inducing a simple $\mathfrak{b}$-representation with weight $\lambda$ yields a universal highest weight module, namely the Verma module with highest weight $\lambda$, which lies in the Bernstein--Gelfand--Gelfand category $\mathcal{O}$ (\cite[§1.3]{humphreys2008representations}). The category $\mathcal{O}$ decomposes into a direct sum of blocks and each block is equivalent to the module category of a quasihereditary algebra. 

Quasihereditary algebras are abundant in representation theory and their ubiquity goes far beyond modeling certain Lie-theoretical contexts. Other examples of quasihereditary algebras include Schur algebras and generalisations of these (\cite{MR866778,MR916172,MR1048073,MR2995138,MR3175171}), but also all algebras of global dimension at most two (\cite{MR987824}) and, in particular, hereditary algebras and Auslander algebras. Quasihereditary algebras are defined by the existence of special quotients of the projective indecomposable modules, called the standard modules, which form an exceptional collection of objects. These correspond to the Verma modules in the blocks of category $\mathcal{O}$ and to the Weyl modules in the case of Schur algebras.

The concept of Borel subalgebra has been generalised to arbitrary quasihereditary algebras by Koenig (\cite{MR1362252}). By analogy with the Lie-theoretical setting, a subalgebra $B$ of a quasihereditary algebra $A$ with indexing poset $(\Phi,\unlhd)$ is said to be an \emph{exact Borel subalgebra} if the following conditions hold:
\begin{enumerate}
\item the induction functor $A\otimes_B -$ is exact;
\item there exists a bijection between the isoclasses of simple $B$-modules and the elements of the poset $(\Phi,\unlhd)$ that turns $B$ into a quasihereditary algebra with respect to $(\Phi,\unlhd)$;
\item all the standard $B$-modules are simple and the induction functor maps the simple $B$-module labelled by $i\in \Phi$ to the standard $A$-module indexed by $i$.
\end{enumerate}

Assume, from now onwards, that we are working over some fixed algebraically closed field. Call two quasihereditary algebras \emph{equivalent} if the corresponding categories of modules filtered by standard modules are equivalent: this essentially means that the two algebras have the same quasihereditary structure. Equivalent quasihereditary algebras are, in particular, Morita equivalent (\cite[Theorem 2]{MR1211481}). 

Generally, it is not the case that every quasihereditary algebra has an exact Borel subalgebra (\cite[Example 2.3]{MR1362252}, \cite[Appendix A.3]{MR3228437}). However, every quasihereditary algebra is equivalent to some other quasihereditary algebra that has an exact Borel subalgebra. This remarkable result was established by Koenig, K\"ulshammer and Ovsienko in \cite{MR3228437} using $A_{\infty}$-technology and techniques from the theory of bocses. Concretely, it was shown in \cite{MR3228437} that  the right algebra $R_{\mathfrak{B}}$ of a directed bocs $\mathfrak{B}=(B,W,\mu,\varepsilon)$ has a natural quasihereditary structure so that the underlying algebra $B$ is an exact Borel subalgebra of $R_{\mathfrak{B}}$. In addition, Koenig, K\"ulshammer and Ovsienko proved that an algebra is quasihereditary if and only if it is equivalent to the right algebra of a directed bocs.

It turns out that Kleiner and Ro\u{\i}ter's regularisation techniques may be applied to the directed bocses obtained in \cite{MR3228437}, so these can always be assumed to be regular (see \cite{doi:10.1112/blms.12331}, namely Proposition 3.10, but also Theorem 3.13 and the paragraph after Remark 3.5). Regular directed bocses are closely related to a certain type of exact Borel subalgebras defined in \cite{doi:10.1112/blms.12331}, which are also called regular. A \emph{regular} exact Borel subalgebra $B$ of a quasihereditary algebra $(A,\Phi,\unlhd)$ is characterised by the existence of isomorphisms
\[\Ext{B}{n}{L}{L'}\rightarrow\Ext{A}{n}{A\otimes_B L}{A\otimes_B L'}, \]
induced by the exact functor $A\otimes_B -$, for every choice of simple $B$-modules $L$ and $L'$ and every $n\geq 1$. Not only are regular exact Borel subalgebras homologically well behaved, but they also appear to possess some sort of intrinsic uniqueness feature to be discussed later on in this introduction. By combining results from \cite{doi:10.1112/blms.12331} and \cite{manuela}, the main theorem of \cite{MR3228437} can be enhanced as follows.

\begin{customthm}{R}[{\cite{MR3228437,doi:10.1112/blms.12331,manuela}}]
	\label{thm:r}
	Let $(A,\Phi,\unlhd)$ be a quasihereditary algebra.
\begin{enumerate}
	\item There exists a regular directed bocs whose right algebra is equivalent to $(A,\Phi,\unlhd)$.
	\item The algebra $(A,\Phi,\unlhd)$ contains a regular exact Borel subalgebra $B$ if and only if $A$ coincides with the right algebra of a regular directed bocs $\mathfrak{B}=(B,W,\mu,\varepsilon)$.
\end{enumerate}
\end{customthm}

This consequential result asserts, in particular, that there always exists at least one algebra in each equivalence class of quasihereditary algebras that is guaranteed to have a regular exact Borel subalgebra. Nevertheless, Theorem \ref{thm:r} raises a great deal of questions which are left unanswered.
\begin{description}
\item[(A) Description of the `good' representatives of {$[(A,\Phi,\unlhd)]$}] \emph{\\Given a quasi\-hered\-i\-tary algebra $(A,\Phi,\unlhd)$, how can we construct an equivalent algebra that contains a (regular) exact Borel subalgebra? Which algebras in the equivalence class $[(A,\Phi,\unlhd)]$ of $(A,\Phi,\unlhd)$ have a regular exact Borel subalgebra, and how many are there? Is there any quasihereditary algebra in $[(A,\Phi,\unlhd)]$ that contains a basic regular exact Borel subalgebra and, if so, is it unique?}

\item[(B) Criteria to identify the `good' quasihereditary algebras] \emph{\\How to decide whether a quasihereditary algebra has a regular exact Borel subalgebra? How to single out the quasihereditary algebras that have a basic regular Borel subalgebra?}

\item[(C) Numerical description of the regular exact Borel subalgebras] \emph{\\In case $B$ is a regular exact Borel subalgebra of some quasihereditary algebra equivalent to $(A,\Phi,\unlhd)$, is it possible to determine information about $B$ (e.g.~its Cartan matrix) without knowing $B$ explicitly?}
	
\item[(D) Characterisation of the best-case scenario] \emph{\\When does it happen that every quasihereditary algebra in an equivalence class $[(A,\Phi,\unlhd)]$ contains a regular exact Borel subalgebra? When does a basic quasihereditary algebra admit a regular exact Borel subalgebra?}

\end{description}

In this paper, we provide answers to all the questions above by using elementary linear algebra and conventional data about quasihereditary algebras. In particular, our methods do not require any calculations with $A_{\infty}$-algebras. Moreover, by applying part (2) of Theorem \ref{thm:r}, our findings may be rephrased as statements about regular directed bocses.

We proceed to describe our main contributions. For this, some minimal notation will be needed. Given a quasihereditary algebra $(A,\Phi,\unlhd)$ and $i\in \Phi$, denote the standard $A$-module with simple top $L_i$ by $\Delta_i$, let $\nabla_i$ be the costandard module with socle $L_i$ and write $[X:L_i]$ for the multiplicity of the simple $L_i$ in the composition series of $X$. The notation $\mathcal{F}(\Delta)$ (resp.~$\mathcal{F}(\nabla)$) will be used for the category of modules filtered by standard modules (resp.~by costandard modules). Occasionally, we shall decorate the simples, standards or costandards with a superscript to indicate the ambient algebra. 

 Our strategy is as follows. To every quasihereditary algebra $(A,\Phi,\unlhd)$, we associate a special matrix $V_{[(A,\Phi,\unlhd)]}=(v_{ij})_{i,j\in\Phi}$. The matrix $V_{[(A,\Phi,\unlhd)]}$ can be computed through a recursive algorithm described in Theorem \ref{thm:algo1} which takes as input the composition factors of the standard and costandard $A$-modules (that is, $[\Delta_i:L_j]$ and $[\nabla_i:L_j]$ for every $i,j\in\Phi$) and also the dimension of the $\HOM$-spaces between standard $A$-modules (i.e.~$\dim(\Hom{A}{\Delta_i}{\Delta_j})$ for $i,j\in\Phi$). The distinguished matrix $V_{[(A,\Phi,\unlhd)]}$ is therefore an invariant of the equivalence class $[(A,\Phi,\unlhd)]$. Furthermore, $V_{[(A,\Phi,\unlhd)]}$ can be realised as a lower triangular matrix with nonnegative entries and with ones on the diagonal, hence $V_{[(A,\Phi,\unlhd)]}$ is invertible. The sum of the entries in each row of $V_{[(A,\Phi,\unlhd)]}=(v_{ij})_{i,j\in\Phi}$ shall be recorded in a sequence, denoted by $l_{[(A,\Phi,\unlhd)]}=(l_{i})_{i\in\Phi}$. We remark that the sequence $l_{[(A,\Phi,\unlhd)]}$ may also be obtained from a recursive formula, described in Corollary \ref{cor:algo1}, which avoids the computation of the matrix $V_{[(A,\Phi,\unlhd)]}$. Using the matrix $V_{[(A,\Phi,\unlhd)]}$ and the sequence $l_{[(A,\Phi,\unlhd)]}$, we prove the following result.
\begin{customthm}{A}[Part of Theorem \ref{thm:main} and Corollary \ref{cor:mainprimitive}]
	\label{thm:a}
	Let $(A,\Phi,\unlhd)$ be a quasihereditary algebra with projective indecomposable modules $P_i$, $i\in \Phi$. Consider the associated matrix $V_{[(A,\Phi,\unlhd)]}=(v_{ij})_{i,j\in\Phi}$ and the sequence $l_{[(A,\Phi,\unlhd)]}=(l_i)_{i\in\Phi}$. The following hold:
	\begin{enumerate}
		\item for every sequence of positive integers $(k_i)_{i\in \Phi}$ there exists a quasihereditary algebra $(R,\Phi,\unlhd)\in[(A,\Phi,\unlhd)]$ which contains a regular exact Borel subalgebra $B$ satisfying $\dim L_i^B = k_i$ for every $i \in \Phi$;
		\item if $(R,\Omega,\preceq)\in[(A,\Phi,\unlhd)]$ has a regular exact Borel subalgebra $B$, then it is possible to relabel the simples over $R$ by elements of $\Phi$, so that $(R,\Omega,\preceq)$ is equivalent and isomorphic to $(\End{A}{\bigoplus_{i\in \Phi}P_i^{m_i}}\op,\Phi,\unlhd)$ with $m_i=\sum_{j \in \Phi}v_{ij}\dim L_i^B$;
		\item the algebra $R_{[(A,\Phi,\unlhd)]}=\End{A}{\bigoplus_{i \in \Phi}P_i^{l_i}}\op$ is quasihereditary with respect to $(\Phi,\unlhd)$ and it only depends on the equivalence class $[(A,\Phi,\unlhd)]$, concretely, on the composition factors of the standard and costandard modules and on the dimension of $\HOM$-spaces between standard modules;
		\item up to isomorphism of algebras, $(R_{[(A,\Phi,\unlhd)]}, \Phi,\unlhd)$ is the unique quasihereditary algebra in $[(A,\Phi,\unlhd)]$ that contains a basic regular exact Borel subalgebra.
	\end{enumerate}
\end{customthm}
By combining parts (1) and (2) of Theorem \ref{thm:a}, we conclude that the isomorphism classes of quasihereditary algebras in $[(A,\Phi,\unlhd)]$ having a regular exact Borel subalgebra are in one-to-one correspondence with the sequences of positive integers $(k_i)_{i\in\Phi}$. Furthermore, the representatives of $[(A,\Phi,\unlhd)]$ that contain a regular exact Borel subalgebra can be computed through the recipe in (2) and, by (4), there exists essentially one algebra in $[(A,\Phi,\unlhd)]$ that contains a basic regular exact Borel subalgebra. Theorem \ref{thm:a} therefore solves all the problems raised in (A).

From Theorem \ref{thm:a}, we deduce a numerical criterion to answer the questions in (B).
\begin{customthm}{B}[Theorem \ref{thm:mainsuper}]
	\label{thm:b}
Let $(A,\Phi,\unlhd)$ be a quasihereditary algebra. Consider the matrix $V_{[(A,\Phi,\unlhd)]}=(v_{ij})_{i,j \in \Phi}$ and the sequence $l_{[(A,\Phi,\unlhd)]}=(l_i)_{i\in \Phi}$. The following hold:
	\begin{enumerate}
		\item $(A,\Phi,\unlhd)$ has a regular exact Borel subalgebra if and only if all the entries of the unique solution of the system $V_{[(A,\Phi,\unlhd)]}x=(\dim L_i^A)_{i\in\Phi}$ are positive integers;
		\item if $(A,\Phi,\unlhd)$ has a regular exact Borel subalgebra $B$, then $(\dim L_i^B)_{i\in\Phi}$ is the unique solution of the system $V_{[(A,\Phi,\unlhd)]}x=(\dim L_i^A)_{i\in\Phi}$ and $[\Rest(L_i^A):L_j^B]=v_{ij}$, where $\Rest: \MOD{A}\rightarrow\MOD{B}$ denotes the restriction functor;
		\item the dimension of the simple modules over a regular exact Borel subalgebra of $(A,\Phi,\unlhd)$ is univocally determined by the dimension of the simple modules over $A$, by the composition factors of the standard and costandard $A$-modules and by the dimension of the $\HOM$-spaces between standard modules;
		\item $(A,\Phi,\unlhd)$ has a basic regular exact Borel subalgebra if and only if $\dim L_i^A=l_i$ for every $i\in \Phi$.
	\end{enumerate}
\end{customthm}
We also show that all regular exact Borel subalgebras of quasihereditary algebras belonging to the same equivalence class have the same Cartan matrix. An explicit description of the Cartan matrix is given in Theorem \ref{thm:c}. This addresses the questions raised in (C).

\begin{customthm}{C}[Part of Theorem \ref{thm:mainthm2}]
	\label{thm:c}
	Let $(A,\Phi,\unlhd)$ be a quasihereditary algebra and consider the associated matrix $V_{[(A,\Phi,\unlhd)]}$. Assume that $B$ is a regular exact Borel subalgebra of some quasihereditary algebra equivalent to $(A,\Phi,\unlhd)$. The Cartan matrix of $B$ is given by
	\[(D_{[(A,\Phi, \unlhd)]}^{\nabla} V_{[(A,\Phi, \unlhd)]})^T,\]
	where $D_{[(A,\Phi, \unlhd)]}^\nabla$ denotes the $\nabla$-decomposition matrix $\left([\nabla_i^A:L_j^A]\right)_{i,j\in\Phi}$ of $(A,\Phi,\unlhd)$. In particular, the Cartan matrix of $B$ is completely determined by the composition factors of the standard and costandard $A$-modules and by the dimension of the $\HOM$-spaces between standard modules, so it only depends on $[(A,\Phi,\unlhd)]$.
\end{customthm}

As announced in \cite[Theorem 4.26]{kulshammer2016bocs}, a bijection between the isomorphism classes of regular directed bocses $\mathfrak{B}=(B,W,\mu,\varepsilon)$ with $B$ basic and the equivalence classes of quasihereditary algebras, mapping a bocs $\mathfrak{B}$ to the equivalence class of its right algebra $R_{\mathfrak{B}}$, shall be provided in upcoming work of K\"ulshammer and Miemietz. We have not been able to prove that every two regular directed bocses $\mathfrak{B}=(B,W,\mu,\varepsilon)$ and $\mathfrak{B}'=(B',W',\mu',\varepsilon')$ with $B$ and $B'$ basic and $R_{\mathfrak{B}}$ and $R_{\mathfrak{B}'}$ equivalent must be isomorphic. However, as a consequence of Theorem \ref{thm:c}, we show, in Corollary \ref{cor:maincor22}, that any two such bocses share a significant amount of information, namely $B$ and $B'$ must have the same Cartan matrix and the same dimension, $R_{\mathfrak{B}}$ and $R_{\mathfrak{B}'}$ have to be isomorphic and the dimension of bimodules $W$ and $W'$ also coincides. Our work therefore corroborates K\"ulshammer and Miemietz's research about the uniqueness of regular directed bocses over basic algebras.

Finally, we investigate when all quasihereditary algebras in a given equivalence class have a regular exact Borel subalgebra. This is connected to the problem of determining which basic quasihereditary algebras contain a regular exact Borel subalgebra. 
\begin{customthm}{D}[Theorem \ref{thm:basicqh}]
	\label{thm:d}
	Let $(A,\Phi,\unlhd)$ be a quasihereditary algebra. The following conditions are equivalent:
	\begin{enumerate}
		\item every quasihereditary algebra in $[(A,\Phi,\unlhd)]$ has a regular exact Borel subalgebra;
		\item the sequence $l_{[(A,\Phi,\lhd)]}$ is constant and equal to one;
		\item $V_{[(A,\Phi,\lhd)]}$ is the identity matrix;
		\item for every $i \in \Phi$, $\Delta_i$ is a right $\mathcal{F}(\Delta)$-approximation of the simple $A$-module $L_i$ (meaning that every morphism $X\rightarrow L_i$ with $X\in \mathcal{F}(\Delta)$ factors through the epic $\pi_i: \Delta_i \twoheadrightarrow L_i$);
		\item $\Rad{\Delta_i}$ belongs to $\mathcal{F}(\nabla)$ for every $i \in \Phi$;
		\item the map $\Ext{A}{1}{X}{\pi_i}:\Ext{A}{1}{X}{\Delta_i}\rightarrow\Ext{A}{1}{X}{L_i}$, where $\pi_i$ denotes the epic from $\Delta_i$ to $L_i$, is an isomorphism for every $X$ in $\mathcal{F}(\Delta)$ and every $i\in\Phi$;
		\item the map $\Ext{A}{1}{\Delta_j}{\pi_i}:\Ext{A}{1}{\Delta_j}{\Delta_i}\rightarrow\Ext{A}{1}{\Delta_j}{L_i}$, where $\pi_i$ denotes the epic from $\Delta_i$ to $L_i$, is an isomorphism for every distinct $i,j\in\Phi$ satisfying $j\unlhd i$.
	\end{enumerate}
	Assuming that $A$ is basic, then the algebra $(A,\Phi,\unlhd)$ contains a regular exact Borel subalgebra if and only if it contains a basic regular exact Borel subalgebra if and only if one of the equivalent conditions (1) to (7) holds.
\end{customthm}

Notice that Theorem \ref{thm:d} solves the problems in (D). We conclude the paper with an application of Theorem \ref{thm:d}. To be precise, we prove, in Proposition \ref{prop:last}, that the Ringel dual of the dual extension of the incidence algebra of a tree always has a regular exact Borel subalgebra.

The layout of this paper is as follows. Section \ref{sec:background} contains background on quasihereditary algebras, exact Borel subalgebras and bocses. Section \ref{sec:rel} delves deeper into the connection between bocses, quasihereditary algebras and exact Borel subalgebras. In particular, it discusses work from \cite{MR3228437} and \cite{doi:10.1112/blms.12331} that is essential to contextualise and to deduce our main results. Section \ref{sec:rel} also contains auxiliary lemmas concerning bocses with an epic counit (namely, Lemmas \ref{lem:lemma1}, \ref{lem:lemma15}, \ref{lem:lemma2}) which are crucial to derive our main theorems. Section \ref{sec:goodrepresentative} is the core of the paper. There, we devise a recursive method that allows us to describe the composition factors of restrictions of $R_{\mathfrak{B}}$-modules for any regular directed bocs $\mathfrak{B}$: this essentially corresponds to Theorem \ref{thm:algo1}. The majority of our main results, namely Theorems \ref{thm:main}, \ref{thm:mainsuper} and \ref{thm:mainthm2}, and Corollaries \ref{cor:mainprimitive} and \ref{cor:maincor22} (which include Theorems \ref{thm:a}, \ref{thm:b} and \ref{thm:c}), are proved in Section \ref{sec:goodrepresentative} and are a by-product of Theorem \ref{thm:algo1}. Section \ref{sec:last} is concerned with basic quasihereditary algebras containing a regular exact Borel subalgebra. A description of this class of algebras is obtained in Theorem \ref{thm:basicqh} (Theorem \ref{thm:d}) and a large class of examples is provided by Proposition \ref{prop:last}.

\subsection{Notation and conventions}
Throughout this paper, $K$ will denote a fixed algebraically closed field. Unless explicitly stated otherwise, the word ‘algebra’ will mean a finite-dimensional $K$-algebra and all modules are assumed to be finite-dimensional left modules. 

Given an algebra $A$, we shall denote the category of finite-dimensional left $A$-modules by $\Mod{A}$; this is a full subcategory of the category $\MOD{A}$ of all (possibly infinite-dimensional) left $A$-modules. 

The isomorphism classes of the simple $A$-modules may be labelled by the elements of a finite set $\Phi$. Denote the simple $A$-modules by $L_i$ or $L_i^A$, $i\in \Phi$, and use the notation $P_i$ or $P_i^A$ (resp.~$Q_i$ or $Q_i^A$) for the projective cover (resp.~injective hull) of $L_i$. Write $\{e_i\mid i \in \Phi\}$ for a complete irredundant set of primitive idempotents in $A$ with $A e_i \cong P_i$. Finally, denote the multiplicity of the simple $L_i$ as a composition factor of a module $X$ by $[X:L_i]$ and let $\ell (X)$ be the length of $X$.

Given a poset $(\Phi,\unlhd)$ and $i,j\in \Phi$, write $i \lhd j$ if $i\unlhd j$ and $i\neq j$.

The arrows in a quiver shall be composed from right to left.

\subsection{Acknowledgments}
I would like to thank Julian K\"ulshammer and Steffen Koenig for many insightful discussions and for comments and suggestions on preliminary versions of this paper. I am deeply grateful to Julian K\"ulshammer for pointing out to me Examples \ref{ex:third} and \ref{ex:last}, which emerged from his work with Agnieszka Bodzenta. Example \ref{ex:last} inspired Proposition \ref{prop:last}.

I acknowledge the support of the Deutsche Forschungsgemeinschaft (DFG) through the grant KO 1281/18.

\section{Preliminaries}

\label{sec:background}

In this section we provide necessary background material on quasihereditary algebras, exact Borel subalgebras and bocses.

\subsection{Quasihereditary algebras}

Assume that $(\Phi,\unlhd)$ is an indexing poset for the simple modules over an algebra $A$. Denote by $\Delta_i$ or $\Delta_i^A$ the largest quotient of the projective indecomposable $P_i$ whose composition factors are all of the form $L_j$ with $j\unlhd i$. Call $\Delta_i$ the \emph{standard module} with label $i\in\Phi$. Dually, use the notation $\nabla_{i}$ or $\nabla_i^A$ for the \emph{costandard module} with label $i$, i.e.~let $\nabla_{i}$ be the largest submodule of $Q_i$ with all composition factors of the form $L_j$, with $j\unlhd i$. 

Denote by $\Delta$ (resp.~$\nabla$) the set of all standard $A$-modules (resp.~all costandard $A$-modules). Let $\mathcal{F}\B{\Delta}$ be the category of all $A$-modules that have a \emph{$\Delta$-filtration}, that is, a filtration whose subquotients are isomorphic to modules in $\Delta$. The category $\mathcal{F}(\nabla)$ is defined in a similar manner.

There are multiple equivalent ways of defining a quasihereditary algebra. Classic references include the pioneering articles of Cline, Parshall and Scott (\cite{MR933417,Moosonee,clineparshallscott}) and the papers \cite{MR987824,MR1211481} by Dlab and Ringel.
\begin{defi}
\label{defi:qh}
An algebra $A$ is \emph{quasihereditary} with respect to a poset $(\Phi, \unlhd)$ indexing all pairwise nonisomorphic simple $A$-modules if the following conditions hold for every $i \in \Phi$:
\begin{enumerate}
\item $[\Delta_i:L_i]=1$;
\item $P_i$ has a $\Delta$-filtration;
\item if $\Ext{A}{1}{\Delta_i}{\Delta_j}\neq 0$, then $i\lhd j$, for any choice of $j\in \Phi$.
\end{enumerate}
\end{defi}
\begin{rem}
Condition (1) in Definition \ref{defi:qh} amounts to saying that the endomorphism algebra $\End{A}{\Delta_i}$ is isomorphic to $K$ and condition (2) may be replaced by
\begin{enumerate}
\item[(2$'$)] $Q_i$ has a $\nabla$-filtration.
\end{enumerate}
\end{rem}

The multiplicity of a standard module as a subquotient in a $\Delta$-filtration of a module $X$ is independent of the choice of the $\Delta$-filtration (\cite[Lemma $2.4$]{MR1211481}). A similar statement holds for modules in $\mathcal{F}(\nabla)$. Denote by $(X: \Delta_i)$ the multiplicity of $\Delta_i$ as subquotient in a $\Delta$-filtration of $X$ in $\mathcal{F}(\Delta)$. Define $(X:\nabla_i)$ for $X$ in $ \mathcal{F}(\nabla)$ in an analogous way.

Dlab and Ringel's Standardisation Theorem (\cite[Theorem 2]{MR1211481}) basically claims that the category of $\Delta$-filtered modules comprises all the essential information about a quasihereditary algebra.
\begin{defi}
	\label{defi:ultimo}
Two quasihereditary algebras are said to be \emph{equivalent} if the corresponding categories of $\Delta$-filtered modules are equivalent.
\end{defi}
\begin{rem}
Equivalent quasihereditary algebras have the same quasihereditary structure. The Standardisation Theorem (\cite[Theorem 2]{MR1211481}) implies that equivalent quasihereditary algebras are Morita equivalent. 
\end{rem}
\begin{rem}
The notion introduced in Definition \ref{defi:ultimo} is an equivalence relation in the class of all quasihereditary algebras over the fixed field $K$. We denote the equivalence class of a quasihereditary algebra $(A,\Phi, \unlhd)$ by $[(A,\Phi, \unlhd)]$.
\end{rem}
\begin{rem}
\label{rem:minimaladapted}
If $(A,\Phi,\unlhd)$ is a quasihereditary algebra, any refinement of the poset $(\Phi,\unlhd)$ gives rise to an equivalent quasihereditary algebra (see Lemmas $2.3$ and $2.12$ in \cite{flores2020combinatorics}, based on \cite[pp.~3--4]{MR1211481} and \cite[Proposition $1.4.12$]{thesis}). One may want to remove relations from $(\Phi,\unlhd)$ in order to obtain a subposet $(\Phi,\unlhd')$ of $(\Phi,\unlhd)$ with respect to which $A$ has the same quasihereditary structure. According to \cite[Lemma 2.6]{flores2020combinatorics}, there exists a unique minimal subposet $(\Phi,\unlhd')$ of $(\Phi,\unlhd)$ for which $(A,\Phi,\unlhd')$ and $(A,\Phi,\unlhd)$ are quasihereditary and equivalent; this is called the \emph{minimal adapted poset} of $(A,\Phi,\unlhd)$. Up to isomorphism of posets, the minimal adapted poset is an invariant of the equivalence class of a quasihereditary algebra (see \cite[Proposition 2.9]{flores2020combinatorics} and also \cite[Definition 1.2.5]{coulembier2019classification}). Hence, the choice of a poset is somehow redundant and the minimal adapted poset of a quasihereditary algebra is, to some extent, the most canonical option.
\end{rem}

\subsection{Exact Borel subalgebras}
\label{subsec:ebsubalgebras} 

The notion of exact Borel subalgebra of a quasihereditary algebra was introduced by Koenig in \cite{MR1362252} and it emulates some of the key properties of Borel subalgebras of complex semisimple Lie algebras. 

\begin{defi}[\cite{MR1362252}]
A subalgebra $B$ of a quasihereditary algebra $(A,\Phi, \unlhd)$ is an \emph{exact Borel subalgebra} of $A$ if the following hold:
\begin{enumerate}
\item the induction functor $A\otimes_B - :\MOD{B} \rightarrow \MOD{A}$ is exact;
\item $\Phi$ is an indexing set for the isomorphism classes of simple $B$-modules and $B$ is a quasihereditary algebra with respect to $(\Phi, \unlhd)$ having simple standard modules;
\item $A\otimes_B L_i^B =\Delta_i^A$ for every $i \in \Phi$.
\end{enumerate}
\end{defi}
Exact Borel subalgebras often satisfy additional properties. We will be focusing on quasihereditary algebras that have a so-called regular exact Borel subalgebra.
\begin{defi}[\cite{doi:10.1112/blms.12331}]
\label{defi:borelprops} 
An exact Borel subalgebra $B$ of a quasihereditary algebra $(A,\Phi, \unlhd)$ is \emph{regular} if the morphisms
\[\
\Ext{B}{n}{L^B_i}{L^B_j} \longrightarrow \Ext{A}{n}{A\otimes_B L^B_i}{A\otimes_B L^B_j},
\]
induced by the exact functor $A\otimes_B - :\MOD{B} \rightarrow \MOD{A}$, are isomorphisms for every $n\geq 1$ and every $i,j\in \Phi$.
\end{defi}

\subsection{Bocses} 
The research carried out in \cite{MR3228437} provided a novel perspective on quasihereditary algebras and revealed the importance of bocses in understanding quasihereditary algebras and their exact Borel subalgebras.

A bocs is like a coalgebra whose ring of scalars does not necessarily act centrally. In this subsection, we introduce the notion of a bocs and define its right algebra.
\begin{defi}
A \emph{bocs} is a quadruple $\mathfrak{B}=(B,W,\mu, \varepsilon)$ consisting of an algebra $B$ and a $B$-$B$-bimodule $W$ (possibly infinite-dimensional over $K$), together with a $B$-$B$-bilinear coassociative comultiplication $\mu: W \rightarrow W\otimes_B W$ and a $B$-$B$-bilinear counit $\varepsilon$. In other words, $\mu$ and $\varepsilon$ are morphisms of $B$-$B$-bimodules for which the following diagrams commute:
\[
\begin{tikzcd}
W \ar[r, "\mu"] \ar[d,"\mu" swap] & W\otimes_B W \ar[d, "\mu\otimes_B1_W"] \\
W\otimes_B W \ar[r,"1_W\otimes_B \mu " swap] & W\otimes_B W\otimes_B W
\end{tikzcd},\quad
\begin{tikzcd}
W\otimes_B W \ar[d, "1_W \otimes_B \varepsilon" swap] & W \ar[r, "\mu"] \ar[l, "\mu" swap] \ar[d,"1_W"] & W\otimes_B W \ar[d, "\varepsilon\otimes_B 1_W"] \\
W\otimes_B B \ar[r,"m_R^W" swap] & W &B \otimes_BW \ar[l,"m_L^W"]
\end{tikzcd}.
\]
Here, the (bijective) maps $m_R^W$ and $m_L^W$ denote, respectively, the right and left multiplication by elements in $B$.  
\end{defi}

To every bocs $\mathfrak{B}=(B,W,\mu, \varepsilon)$ one may canonically associate two (possibly infinite-dimensional) algebras over $K$: the right and the left algebra of $\mathfrak{B}$. Only the definition of right algebra of a bocs will be needed in this paper. 

\begin{defi}[{\cite{assforbocses}}]
The \emph{right algebra} $R_{\mathfrak{B}}$ of a bocs $\mathfrak{B}=(B,W,\mu, \varepsilon)$ consists of the $B$-$B$-bimodule $\Hom{B}{W}{B}$ endowed with the multiplication $s\circ_{\mathfrak{B}}t$ for $s,t \in \Hom{B}{W}{B}$ given by the composition
\[
\begin{tikzcd}
W \ar[r, "\mu"] & W\otimes_B W  \ar[r, "1_W\otimes_B s"] &W\otimes_B B \ar[r, "m_R^W"] & W \ar[r, "t"] & B
\end{tikzcd}.
\]
\end{defi}
\begin{rem}
A standard verification shows that $\varepsilon$ is the identity of $R_{\mathfrak{B}}$. 
\end{rem}
\begin{rem}
\label{rem:inclusion}
Let $\mathfrak{B}=(B,W,\mu, \varepsilon)$ be a bocs. Note that $\End{B}{B}$ is isomorphic to $B\op$. Through this identification, the map $\Hom{B}{\varepsilon}{B}:\End{B}{B}\rightarrow \Hom{B}{W}{B}$ gives rise to an algebra homomorphism $\iota_{\mathfrak{B}}:B \rightarrow R_{\mathfrak{B}}$. The morphism $\iota_{\mathfrak{B}}$ maps $b\in B$ to $b \varepsilon$, where the $B$-action comes from the natural $B$-$B$-bimodule structure of $R_{\mathfrak{B}}$ (see \cite[§$1.1'$]{assforbocses}). Observe that $\iota_{\mathfrak{B}}=\Hom{B}{\varepsilon}{B}\op$ is a monomorphism whenever $\varepsilon$ is an epic, so, in this case, $B$ may be regarded as a subalgebra of $R_{\mathfrak{B}}$.
\end{rem}

We will be working with bocses $\mathfrak{B}=(B,W,\mu, \varepsilon)$ for which the bimodule $W$ is finite-dimensional over $K$, hence the right algebras studied in this paper are actually finite-dimensional.

\section{How bocses relate to quasihereditary algebras and exact Borel subalgebras}
\label{sec:rel}
Some basic familiarity with the relation between directed bocses, quasihereditary algebras and exact Borel subalgebras is required in order to set up the framework for our main results. Following the work in \cite{MR3228437,doi:10.1112/blms.12331}, this section focuses on special features that a bocs may have and how these are connected to properties of exact Borel subalgebras. The key auxiliary lemmas needed to prove our main results are also included in this section.

\subsection{More on bocses}
\label{subsec:moreonbocses}
The algebra morphism $\iota_{\mathfrak{B}}:B \rightarrow R_{\mathfrak{B}}$ described in Remark \ref{rem:inclusion} turns every $R_{\mathfrak{B}}$-module into a $B$-module by restriction of the action. The categories of modules over $R_{\mathfrak{B}}$ and $B$ are therefore connected via the adjoint triple consisting of the induction, restriction and coinduction functors:
\begin{equation}
\label{eq:adointtriple}
\begin{tikzcd}[column sep=huge]
 \MOD{B} \ar[r, bend left, "R_{\mathfrak{B}}\otimes_B-",""{name=Aa}]  \ar[r, bend right,  "\Hom{B}{R_{\mathfrak{B}}}{-}" swap,""{name=Cc},shift right=1.5ex] & \MOD{R_{\mathfrak{B}}} \ar[l, "\Rest",""{name=Bb}] \\
\arrow[phantom, symbol=\dashv, from=Aa, to=Bb]
\arrow[phantom, symbol=\dashv, from=Bb, to=Cc, near end]
\end{tikzcd}.
\end{equation}

The next auxiliary results will be needed in Section \ref{sec:goodrepresentative} and will also be used later in this section for the characterisation of certain types of bocses.

\begin{lem}
\label{lem:lemma1}
Let $\mathfrak{B}=(B,W,\mu, \varepsilon)$ be a bocs. Denote by $\eta$ the unit of the adjunction $R_{\mathfrak{B}}\otimes_B - \dashv \Rest$. There exists a commutative diagram of natural transformations
\[
\begin{tikzcd}
1_{\MOD{B}} \ar[d,"\alpha"']\isoarrow{d}\ar[r, "\eta"] & \Rest \circ (R_{\mathfrak{B}}\otimes_B -)\ar[d,"\beta"] \\
\Hom{B}{B}{-} \ar[r, "\Hom{B}{\varepsilon}{-}"] & \Hom{B}{W}{-}
\end{tikzcd},
\]
where $\alpha$ is a natural isomorphism and $\beta$ is a natural transformation which becomes a natural isomorphism whenever $W$ is finite-dimensional and projective as a (left) $B$-module.
\end{lem}
\begin{proof}
Consider $X$ in $\MOD{B}$. The unit $\eta_X$ of the adjoint pair $R_{\mathfrak{B}}\otimes_B - \dashv \Rest$ can be chosen as the map which sends $x \in X$ to $\varepsilon \otimes_B x$ (recall that $\varepsilon$ is the identity of $R_{\mathfrak{B}}$). Let $\alpha_X$ be the function mapping $x \in X$ to the unique morphism $\alpha_X(x) \in \Hom{B}{B}{X}$ for which $(\alpha_X(x))(1)=x$ holds.  Define $\beta_X:\Hom{B}{W}{B}\otimes_B X \rightarrow \Hom{B}{W}{X}$ to be the morphism of $B$-modules satisfying $(\beta_X(s\otimes_B x))(w)=s(w)x$ for every $s\in \Hom{B}{W}{B}$, $x\in X$ and $w \in W$. These assignments give rise to natural transformations $\alpha$ and $\beta$. The diagram in the statement of the lemma commutes, since
\begin{align*}
\Big(\big(\Hom{B}{\varepsilon}{X} \circ \alpha_X\big)(x)\Big)(w)&= \Big(\Hom{B}{\varepsilon}{X} \big(\alpha_X(x) \big)\Big)(w)=\big((\alpha_X(x)) \circ \varepsilon\big)(w) \\
&= (\alpha_X(x))(\varepsilon(w))=\varepsilon(w) x =\big(\beta_X(\varepsilon \otimes_B x)\big)(w)\\
&=\big((\beta_X \circ \eta_X)(x) \big)(w)
\end{align*}
for every $X$ in $\MOD{B}$, $x \in X$ and $w\in W$. It is well known that the natural transformation $\beta$ is a natural isomorphism when $W$ is finitely generated and projective for the left $B$-action: for this, we refer to part (ii) of Proposition 2 in Chapter II, §4.2 of \cite{bourbaki1998algebra}.
\end{proof}
We will be dealing with bocses with an epic counit whose kernel is finite-dimen\-sion\-al and projective as a bimodule. The next two lemmas are concerned with bocses satisfying such properties.
\begin{lem}
\label{lem:lemma15}
Let $\mathfrak{B}=(B,W,\mu, \varepsilon)$ be a bocs with epic counit $\varepsilon$. Assume that $\Phi$ is a labelling set for the simple $B$-modules and that $\Ker{\varepsilon}$ is finite-dimensional and projective as a $B$-$B$-bimodule. Then, the bimodule $\Ker{\varepsilon}$ is isomorphic to $\bigoplus_{k,l \in \Phi}(B e_k \otimes_K e_l B)^{n_{kl}}$, for certain nonnegative integers $n_{kl}$. Furthermore, for every $B$-module $X$, the unit $\eta_X$ of the adjunction $R_{\mathfrak{B}}\otimes_B - \dashv \Rest$ is a monic with injective cokernel. To be precise, the sequence
\[
\begin{tikzcd}
0 \ar[r] & X \ar[r, "\eta_{X}"] & \Rest (R_{\mathfrak{B}}\otimes_B X)\ar[r] & \bigoplus_{l \in \Phi} (Q_l^B)^{n_l(X)} \ar[r]& 0
\end{tikzcd},
\]
with $n_l(X)=\left( \sum_{k\in \Phi}(X:L_k^B)n_{kl}\right) \dim L_l^B$, is exact.
\end{lem}
\begin{proof}
Since $K$ is algebraically closed, every projective indecomposable $B$-$B$-bi\-mo\-du\-le is of the form $B e_k \otimes_K e_l B$ for $k,l\in\Phi$ -- for $B$ basic, this follows, for instance, from Lemma 5.3.8 and Corollary 5.3.10 in \cite{Zimmermann:1952386} and the nonbasic case can then be deduced through Morita theory. Hence $\Ker{\varepsilon}\cong\bigoplus_{k,l \in \Phi}(B e_k \otimes_K e_l B)^{n_{kl}}$ for certain nonnegative integers $n_{kl}$ and $\Ker{\varepsilon}$ is projective as a right and as a left $B$-module. Consider the exact sequence 
\begin{equation}
\label{eq:sescounit}
\begin{tikzcd}
0 \ar[r] & \Ker{\varepsilon} \ar[r] & W \ar[r, "\varepsilon"] & B\ar[r]& 0
\end{tikzcd}.
\end{equation}
Note that \eqref{eq:sescounit} splits in $\MOD{B}$, hence its image through the functor $\Hom{B}{-}{X}$ gives rise to another short exact sequence in $\MOD{B}$. Since $\Ker{\varepsilon}$ is finite-di\-men\-sion\-al, then so is $W$ and the splitting of \eqref{eq:sescounit} also implies that $W$ is projective as a right and as a left $B$-module. Lemma \ref{lem:lemma1} then guarantees the existence of an exact sequence of $B$-modules
\[
\begin{tikzcd}
0 \ar[r] & X \ar[r, "\eta_{X}"] & \Rest (R_{\mathfrak{B}}\otimes_B X)\ar[r] & \Hom{B}{\Ker{\varepsilon}}{X} \ar[r]& 0
\end{tikzcd}.
\]
Using the decomposition of $\Ker{\varepsilon}$ as a direct sum of indecomposable $B$-$B$-bimodules, one easily deduces that $\Ker{\varepsilon}\otimes_B Y$ is a projective (left) $B$-module for every $Y$ in $\Mod{B}$. It then follows from Lemma 2.9 in \cite{doi:10.1112/blms.12331} that the $B$-module $Z_X=\Hom{B}{\Ker{\varepsilon}}{X}$ is injective. Hence, $Z_X\cong \bigoplus_{i\in \Phi} (Q_i^B)^{n_i(X)}$ for certain nonnegative integers $n_{i}(X)$. Observe that
\begin{align*}
\Hom{B}{L_i^B}{Z_X}&\cong \Hom{B}{\Ker{\varepsilon}\otimes_B L_i^B}{X} \\
&\cong  \Hom{B}{\bigoplus_{k,l \in \Phi}(B e_k \otimes_K e_l B\otimes_B L_i^B)^{n_{kl}}}{X} \\
&\cong  \Hom{B}{\bigoplus_{k,l \in \Phi}(B e_k \otimes_K e_l B\otimes_B (Be_i/\Rad{B}e_i))^{n_{kl}}}{X} \\
&\cong  \Hom{B}{\bigoplus_{k\in \Phi}B e_k^{n_{ki}\dim L_i^B}}{X}.
\end{align*}
As a consequence,
\begin{align*}
n_i(X)&=\dim\left(\Hom{B}{L_i^B}{Z_X}  \right)=\dim\left( \Hom{B}{\bigoplus_{k\in \Phi}B e_k^{n_{ki}\dim L_i^B}}{X} \right) \\
&=\left( \sum_{k\in \Phi}(X:L_k^B)n_{ki}\right) \dim L_i^B.
\end{align*}
\end{proof}
It is possible to derive an explicit formula for the integers $n_{kl}$ appearing in the bimodule decomposition of $\Ker{\varepsilon}$ in Lemma \ref{lem:lemma15}. The symbol $\delta_{kl}$ shall henceforth denote the Kronecker delta.
\begin{lem}
\label{lem:lemma2}
Let $\mathfrak{B}=(B,W,\mu, \varepsilon)$ be a bocs with epic counit $\varepsilon$. Assume that $\Phi$ is a labelling set for the simple $B$-modules and that $\Ker{\varepsilon}$ is finite-dimensional and projective as a $B$-$B$-bimodule. Then
\[
d_{ij}=\dim\left( \Ext{B}{1}{ L_j^B}{L_i^B}\right) -\dim\left( \Ext{R_{\mathfrak{B}}}{1}{R_{\mathfrak{B}}\otimes_B L_j^B}{R_{\mathfrak{B}}\otimes_B L_i^B}\right) \geq 0
\]
for every $i,j\in\Phi$ and $\Ker{\varepsilon}$ is isomorphic to $\bigoplus_{i,j \in \Phi}(B e_i \otimes_K e_j B)^{n_{ij}} $ as a $B$-$B$-bimodule, where the integer $n_{ij}$ is given by
\[\frac{\dim \big(\Hom{R_{\mathfrak{B}}}{R_{\mathfrak{B}}\otimes_B L_j^B}{R_{\mathfrak{B}}\otimes_B L_i^B}\big)-\delta_{ij}+d_{ij}}{\dim L_j^B}.\]
\end{lem}
\begin{proof}
By Lemma \ref{lem:lemma15}, $\Ker{\varepsilon}\cong \bigoplus_{k,l \in \Phi}(B e_k \otimes_K e_l B)^{n_{kl}}$ for certain nonnegative integers $n_{kl}$. Lemma \ref{lem:lemma15} also assures the existence of an exact sequence of $B$-modules
\begin{equation}
\label{eq:ses1}
\begin{tikzcd}
0 \ar[r] & L_i^B \ar[r, "\eta_{L_i^B}"] & \Rest (R_{\mathfrak{B}}\otimes_B L_i^B)\ar[r] & Z_i \ar[r]& 0
\end{tikzcd},
\end{equation}
where $Z_i$ is injective and isomorphic to $\bigoplus_{l\in \Phi} (Q_l^B)^{n_{il}\dim L_l^B}$. We apply the functor $\Hom{B}{L_j^B}{-}$ to \eqref{eq:ses1} and obtain the long exact sequence
\[
\begin{tikzcd}[column sep=small]
0 \arrow[r]&  \Hom{B}{L_j^B}{L_i^B} \arrow[r]
    & \Hom{B}{L_j^B}{\Rest (R_{\mathfrak{B}}\otimes_B L_i^B)} \arrow[r]
        \arrow[d, phantom, ""{coordinate, name=Z}]
      & \Hom{B}{L_j^B}{Z_i} \arrow[dll,rounded corners,to path={ -- ([xshift=2ex]\tikztostart.east)
|- (Z) [near end]\tikztonodes
-| ([xshift=-2ex]\tikztotarget.west) -- (\tikztotarget)}] \\
& \Ext{B}{1}{L_j^B}{L_i^B} \arrow[r,"g^{(ij)}"]
    & \Ext{B}{1}{L_j^B}{\Rest (R_{\mathfrak{B}}\otimes_B L_i^B)} \arrow[r]
&0 \,.\end{tikzcd}
\]
Since $R_{\mathfrak{B}}\otimes_B - $ is left adjoint to $\Rest$, then
\[\Hom{B}{L_j^B}{\Rest (R_{\mathfrak{B}}\otimes_B L_i^B)}\cong \Hom{R_{\mathfrak{B}}}{R_{\mathfrak{B}}\otimes_B L_j^B}{R_{\mathfrak{B}}\otimes_B L_i^B}.\]
According to \cite[§$2.1$]{assforbocses}, $R_{\mathfrak{B}}$ is projective as a right module over $B$, hence the Eckmann-Shapiro Lemma (\cite[Corollary 2.8.4]{benson_1991}) implies that
\[
\Ext{B}{1}{L_j^B}{\Rest (R_{\mathfrak{B}}\otimes_B L_i^B)}\cong \Ext{R_{\mathfrak{B}}}{1}{R_{\mathfrak{B}}\otimes_B L_j^B}{R_{\mathfrak{B}}\otimes_B L_i^B}.
\] 
The dimension of $\Hom{B}{L_j^B}{Z_i}$ is given by
\[\dim\left( \Hom{R_{\mathfrak{B}}}{R_{\mathfrak{B}}\otimes_B L_j^B}{R_{\mathfrak{B}}\otimes_B L_i^B}\right) -\delta_{ij}+\dim\left( \Ker{g^{(ij)}}\right)\]
for every $i,j \in \Phi$. On the other hand, we have 
\[
\dim\left(\Hom{B}{L_j^B}{Z_i}\right)=\dim\left(\Hom{B}{L_j^B}{\bigoplus_{l\in \Phi} (Q_l^B)^{n_{il}\dim L_l^B}}\right)
=n_{ij}\dim L_j^B.
\]
\end{proof}
\subsection{Directed bocses and exact Borel subalgebras}
The bocses associated to quasihereditary algebras obtained by Koenig, K\"ulshammer and Ovsienko in \cite{MR3228437} are especially nice: they are directed. 
\begin{defi}[\cite{doi:10.1112/blms.12331,MR3228437,kulshammer2016bocs,MR3800074}]
A bocs $\mathfrak{B}=(B,W,\mu, \varepsilon)$ is \emph{directed} if the following conditions hold:
\begin{enumerate}
\item the counit $\varepsilon$ is epic;
\item $B$ is a quasihereditary algebra with respect to some indexing poset $(\Phi, \unlhd)$ and the standard $B$-modules are simple;
\item $\Ker{\varepsilon}$ is a direct sum of finitely many $B$-$B$-bimodules of the form $Be_j \otimes_K e_i B$, with $i,j \in\Phi$ and $i\lhd j$.
\end{enumerate}
\end{defi}
Directed bocses always give rise to quasihereditary algebras and exact Borel subalgebras.
\begin{thm}[{\cite[Theorem 11.2]{MR3228437}}]
\label{thm:kko1}
If $\mathfrak{B}=(B,W,\mu, \varepsilon)$ is a directed bocs, then $R_{\mathfrak{B}}$ is quasihereditary and $B$ is an exact Borel subalgebra of $R_{\mathfrak{B}}$. Concretely, assuming that $B$ is quasihereditary with respect to $(\Phi,\unlhd)$, then $R_{\mathfrak{B}}$ is also quasihereditary with respect to $(\Phi,\unlhd)$ and has standard modules $\Delta^{R_{\mathfrak{B}}}_i=R_{\mathfrak{B}}\otimes_B L_i^B$. The algebra monomorphism $\iota_{\mathfrak{B}}:B \hookrightarrow R_{\mathfrak{B}}$ in Remark \ref{rem:inclusion} turns $B$ into an exact Borel subalgebra of $R_{\mathfrak{B}}$.
\end{thm}
Conversely, every quasihereditary algebra comes from a bocs.
\begin{thm}[{\cite[Theorem 11.3, proof of Corollary 11.4]{MR3228437}}]
\label{thm:kkodifficultimplication}
Let $(A,\Phi,\unlhd)$ be a quasihereditary algebra. There exists a directed bocs $\mathfrak{B}=(B,W,\mu, \varepsilon)$ such that $R_{\mathfrak{B}}$ is equivalent to $(A,\Phi,\unlhd)$. Specifically, the simple $B$-modules can be labelled in such a way that $B$ is quasihereditary with respect to $(\Phi, \unlhd)$ with simple standard modules, the quasihereditary structure of $(R_{\mathfrak{B}},\Phi,\unlhd)$ is as stated in Theorem \ref{thm:kko1} and $\Delta_i^A$ is mapped to $\Delta_i^{R_{\mathfrak{B}}}$ through some equivalence of categories $\mathcal{F}(\Delta^A)\rightarrow \mathcal{F}(\Delta^{R_{\mathfrak{B}}})$.
\end{thm}
According to \cite[Proposition 3.10]{doi:10.1112/blms.12331} and \cite[Corollary 3.3]{manuela}, the bocses in Theorem \ref{thm:kkodifficultimplication} can be assumed to be regular. The usual definition of regularity for bocses is rather technical and it is tied to language of differential biquivers. For this reason, we shall use an equivalent description of regularity for directed bocses. The next result was communicated to me by Julian K\"ulshammer and is part of his joint work with Vanessa Miemietz. In fact, the statement of Lemma \ref{lem:julianvanessa} is essentially an extension of Lemma 9.4 in \cite{MR3800074} to the case where the underlying algebra of the bocs is not necessarily basic. 
\begin{lem}
\label{lem:julianvanessa}
Let $\mathfrak{B}=(B,W,\mu, \varepsilon)$ be a directed bocs. The following assertions are equivalent:
\begin{enumerate}
\item $\mathfrak{B}$ is regular;
\item $\Ext{B}{1}{\bigoplus_{i\in \Phi}L_i^B}{\bigoplus_{i\in \Phi}L_i^B}\cong\Ext{R_{\mathfrak{B}}}{1}{\bigoplus_{i\in \Phi}\Delta^{R_{\mathfrak{B}}}_i}{\bigoplus_{i\in \Phi}\Delta^{R_{\mathfrak{B}}}_i}$ in $\MOD{K}$;
\item $\dim\left( \Ext{B}{1}{L_j^B}{L_i^B}\right)=\dim \left(\Ext{R_{\mathfrak{B}}}{1}{\Delta^{R_{\mathfrak{B}}}_j}{\Delta^{R_{\mathfrak{B}}}_i}\right)$ for every $i,j\in\Phi$;
\item $\Ker{\varepsilon}\cong\bigoplus_{i,j \in \Phi}(B e_i \otimes_K e_j B)^{n_{ij}}$ as a $B$-$B$-bimodule, with 
\[n_{ij}=\frac{\dim \Big(\Hom{R_{\mathfrak{B}}}{\Delta_j^{R_{\mathfrak{B}}}}{\Rad{\Delta_i^{R_{\mathfrak{B}}}}}\Big)}{\dim L_j^B}.
\]
\end{enumerate}
\end{lem}
\begin{proof}
The equivalence between the assertions (1) and (2) is a consequence of \cite[Lemma 3.12]{doi:10.1112/blms.12331} and \cite[Corollary 3.3]{manuela}. The directedness of $\mathfrak{B}$ implies that $\varepsilon$ is epic and that $\Ker{\varepsilon}$ is finite-dimensional and projective as a $B$-$B$-bimodule. The equivalence between the statements (2), (3) and (4) follows then from Lemma \ref{lem:lemma2} and Theorem \ref{thm:kko1} and from the definition of quasihereditary algebra.
\end{proof}

Using Proposition 3.10 in \cite{doi:10.1112/blms.12331} (see also Theorem 3.13 and the paragraph after Remark 3.5 in the same reference) and Corollary 3.3 in \cite{manuela}, Theorem \ref{thm:kkodifficultimplication} can be rephrased as follows.

\begin{thm}
\label{thm:kkorephrased}
Let $(A,\Phi,\unlhd)$ be a quasihereditary algebra. There exists a regular directed bocs $\mathfrak{B}=(B,W,\mu, \varepsilon)$ such that $R_{\mathfrak{B}}$ is equivalent to $(A,\Phi,\unlhd)$. More precisely, the simple $B$-modules can be labelled in a way such that $B$ is quasihereditary with respect to $(\Phi, \unlhd)$ with simple standard modules, the quasihereditary structure of $(R_{\mathfrak{B}},\Phi,\unlhd)$ is as stated in Theorem \ref{thm:kko1} and $\Delta_i^A$ is mapped to $\Delta_i^{R_{\mathfrak{B}}}$ through an equivalence $\mathcal{F}(\Delta^A)\rightarrow \mathcal{F}(\Delta^{R_{\mathfrak{B}}})$. In this setting, $B$ is a regular exact Borel subalgebra of $(R_{\mathfrak{B}},\Phi,\unlhd)$.
\end{thm}

For ease of reference later on in the paper, we assign a special name to the regular directed bocses whose underlying algebra is basic.
\begin{defi}
	\label{def:goodbocs}
	A regular directed bocs $\mathfrak{B}=(B,W,\mu, \varepsilon)$ is \emph{minimal} if the algebra $B$ is basic.
\end{defi}
\section{Good quasihereditary algebras and regular exact Borel subalgebras}
\label{sec:goodrepresentative}

According to Theorem \ref{thm:kkorephrased}, given a quasihereditary algebra $(A,\Phi,\unlhd)$, there exists some regular directed bocs $\mathfrak{B}=(B,W,\mu, \varepsilon)$ such that $R_{\mathfrak{B}}$ is equivalent to $(A,\Phi,\unlhd)$ and $B$ is a regular exact Borel subalgebra of $R_{\mathfrak{B}}$. In this section, we provide formulas to compute the dimension of the simple modules over $R_{\mathfrak{B}}$ when $\mathfrak{B}=(B,W,\mu, \varepsilon)$ is a regular directed bocs and $R_{\mathfrak{B}}$ lies in the equivalence class $[(A,\Phi,\unlhd)]$ of $(A,\Phi,\unlhd)$. We show that, up to isomorphism, such an algebra $R_{\mathfrak{B}}$ is determined by the composition factors of the standard and costandard $A$-modules, by the dimension of the $\HOM$-spaces between standard modules and by the dimension of the simples over $B$. By applying a result from \cite{doi:10.1112/blms.12331}, our considerations about regular directed bocses are translated into statements about regular exact Borel subalgebras. As a highlight in this section, we present a necessary and sufficient criterion for a quasihereditary algebra to have a regular exact Borel subalgebra and show that the Cartan matrix of the regular exact Borel subalgebra is completely determined by the composition factors of the standard and costandard modules and by the dimension of the $\HOM$-spaces between the standard modules. We also deduce that, up to isomorphism, there exists a unique quasihereditary algebra containing a basic regular exact Borel subalgebra in each equivalence class of quasihereditary algebras.

Throughout the rest of the paper we will often be dealing with three quasihereditary algebras simultaneously: these are usually an algebra $(A,\Phi,\unlhd)$ and two algebras $B$ and $R_{\mathfrak{B}}$, where $\mathfrak{B}=(B,W,\mu, \varepsilon)$ is a directed bocs. In order to distinguish the simple modules over the three algebras, we denote the simple $A$-modules by $L_i$ and use the notation $L_i^B$ and $L_i^{R_{\mathfrak{B}}}$ for the simples over $B$ and $R_{\mathfrak{B}}$, respectively. The same logic will be applied when denoting the projective and injective indecomposable modules, as well as the standard and costandard modules over $A$, $B$ and $R_{\mathfrak{B}}$.

\subsection{Composition factors of restricted modules}
Let $\mathfrak{B}=(B,W,\mu, \varepsilon)$ be a directed bocs and suppose that the simple $B$-modules are indexed by $\Phi$. Recall the adjoint triple in \eqref{eq:adointtriple}, associated to $\mathfrak{B}$. We shall start by deriving a recursive formula that describes the composition factors of the $B$-modules $\Rest(L_i^{R_{\mathfrak{B}}})$ for every $i\in\Phi$. When the bocs $\mathfrak{B}$ is also regular, our formula only depends on elementary data about the quasihereditary structure of $R_{\mathfrak{B}}$.

The next result, due to Koenig, will be used a couple of times. We state it for the convenience of the reader.

\begin{thm}[{\cite[part of Theorem A]{MR1362252}}]
\label{thm:koenigborel}
Let $B$ be an exact Borel subalgebra of a quasihereditary algebra $(A,\Phi,\unlhd)$. The restriction functor $\Rest: \MOD{A} \rightarrow \MOD{B}$ gives rise to an isomorphism of $B$-modules $\Rest(\nabla_i)\cong Q_i^B= \nabla_i^B$.
\end{thm}

Given a module $X$ over an algebra $B$, write $[X]$ for its image in the \emph{Grothendieck group} $G(B)$ of $B$ (for further details, we refer to \cite{lang2005algebra}, just before §5 in Chapter III). If $\Phi$ is an indexing set for the isoclasses of simple $B$-modules, the set $\{[L_i^B]\mid i \in \Phi \}$ constitutes a $\Z$-basis of $G (B)$. The coordinates of $[X]$ written as $\Z$-linear combination of the elements in $\{[L_i^B]\mid i \in \Phi \}$ coincide with the composition factors of $X$.

\begin{prop}
\label{prop:dimvectors}
Let $\mathfrak{B}=(B,W,\mu, \varepsilon)$ be a directed bocs and let $\Phi$ be a labelling set for the simple $B$-modules. For every $i \in \Phi$, the $B$-module $\Rest(L_i^{R_{\mathfrak{B}}})$ has simple socle $L_i^B$ and all its remaining composition factors, if any, are of the form $L_j^B$ with $j\lhd i$. Furthermore, the following identities hold for every $i \in \Phi$:
\begin{align*}
[\Rest(L_i^{R_{\mathfrak{B}}})]=&[L_i^B] + \sum_{j,k \in \Phi}z_{ij}[\nabla_j^{R_{\mathfrak{B}}}:L_k^{R_{\mathfrak{B}}}] [\Rest(L_k^{R_{\mathfrak{B}}})]- \sum_{j\in \Phi}[\Rad{\Delta_i^{R_{\mathfrak{B}}}}:L_j^{R_{\mathfrak{B}}}][\Rest (L_j^{R_{\mathfrak{B}}})]
 \\
=& [L_i^B] + \sum_{\substack{j,k \in \Phi \\ k\unlhd j \lhd i}}z_{ij}[\nabla_j^{R_{\mathfrak{B}}}:L_k^{R_{\mathfrak{B}}}] [\Rest(L_k^{R_{\mathfrak{B}}})]  - \sum_{\substack{j \in \Phi \\ j \lhd i}}[\Delta_i^{R_{\mathfrak{B}}}:L_j^{R_{\mathfrak{B}}}][\Rest (L_j^{R_{\mathfrak{B}}})],
\end{align*}
with $z_{ij}$ given by
\[
\dim\left( \Hom{R_{\mathfrak{B}}}{\Delta_j^{R_{\mathfrak{B}}}}{\Rad{\Delta_i^{R_{\mathfrak{B}}}}}\right)+\dim\left( \Ext{B}{1}{L_j^B}{L_i^B}\right)- \dim\left( \Ext{R_{\mathfrak{B}}}{1}{\Delta_j^{R_{\mathfrak{B}}}}{\Delta_i^{R_{\mathfrak{B}}}}\right).
\]
\end{prop}
\begin{proof}
By Theorem \ref{thm:kko1}, the algebra $R_{\mathfrak{B}}$ is quasihereditary with respect to $(\Phi,\unlhd)$ and $B$ is an exact Borel subalgebra of $R_{\mathfrak{B}}$. The exact functor $\Rest$ sends the inclusion of $L_i^{R_{\mathfrak{B}}}$ into $\nabla_i^{R_{\mathfrak{B}}}$ to the inclusion of $\Rest(L_i^{R_{\mathfrak{B}}})$ into $\Rest(\nabla_i^{R_{\mathfrak{B}}})$. Theorem \ref{thm:koenigborel} implies that $\Rest(\nabla_i^{R_{\mathfrak{B}}})\cong Q_i^B=\nabla_i^B$. Consequently, the nonzero module $\Rest(L_i^{R_{\mathfrak{B}}})$ has simple socle $L_i^B$ and all its other composition factors must be of the form $L_j^B$ with $j\lhd i$.

Since the functor $\Rest$ is exact, then
\begin{align*}
[\Rest(L_i^{R_{\mathfrak{B}}})]&=[\Rest({\Delta}_i^{R_{\mathfrak{B}}})]-[\Rest(\Rad{{\Delta}_i^{R_{\mathfrak{B}}}})]=[\Rest(R_{\mathfrak{B}}\otimes_B L_i^B)]-[\Rest(\Rad{{\Delta}_i^{R_{\mathfrak{B}}}})] \\
&=[L_i^B]+\sum_{j\in \Phi}n_j(L_i^B) [Q_j^B]-[\Rest(\Rad{{\Delta}_i^{R_{\mathfrak{B}}}})]\\
&=[L_i^B]+\sum_{j\in \Phi}n_{ij} \dim L_j^B [\Rest({\nabla}_j^{R_{\mathfrak{B}}})]-[\Rest(\Rad{{\Delta}_i^{R_{\mathfrak{B}}}})]
\end{align*}
where the last two equalities follow from Lemma \ref{lem:lemma15} and Theorem \ref{thm:koenigborel}. By Lemma \ref{lem:lemma2}, the product $n_{ij} \dim L_j^B$ coincides with the integer $z_{ij}$ in the statement of the proposition. From the exactness of $\Rest$ it follows that
\begin{gather*}
[\Rest({\nabla}_j^{R_{\mathfrak{B}}})]=\sum_{k\in \Phi}[\nabla_j^{R_{\mathfrak{B}}}:L_k^{R_{\mathfrak{B}}}][\Rest (L_k^{R_{\mathfrak{B}}})],\\ [\Rest(\Rad{\Delta_i^{R_{\mathfrak{B}}}})]=\sum_{k\in \Phi}[\Rad{\Delta_i^{R_{\mathfrak{B}}}}:L_j^{R_{\mathfrak{B}}}][\Rest (L_j^{R_{\mathfrak{B}}})].
\end{gather*}
This proves the first identity in the statement of the proposition. The second identity follows from the properties of quasihereditary algebras.
\end{proof}
Notice that the formula derived in Proposition \ref{prop:dimvectors} is recursive. As input, it takes data corresponding to the quasihereditary structure of $(R_{\mathfrak{B}},\Phi,\unlhd)$, but it also depends on the dimension of the extensions of the simples over the underlying algebra of the bocs. Philosophically speaking, we want to be able to extract as much information as possible from a directed bocs by simply looking at the quasihereditary structure of its right algebra. Taking this into consideration, we shall see that it is sensible to focus on regular directed bocses. Other reasons for restricting our attention to regular directed bocses will hopefully become apparent later on.

\begin{cor}
\label{cor:recursivedimvectors}
Let $(A,\Phi,\unlhd)$ be a quasihereditary algebra and let $\mathfrak{B}=(B,W,\mu, \varepsilon)$ be a regular directed bocs such that $R_{\mathfrak{B}}$ is equivalent to $(A,\Phi,\unlhd)$. The following identity holds for every $i \in \Phi$:
\begin{align}
[\Rest(L_i^{R_{\mathfrak{B}}})]=& [L_i^B] + \sum_{\substack{j,k \in \Phi \\ k\unlhd j \lhd i}}[\nabla_j:L_k] \dim\left( \Hom{A}{\Delta_j}{\Delta_i}\right) [\Rest(L_k^{R_{\mathfrak{B}}})]  \label{eq:dimvectorsformula}\\
 &- \sum_{\substack{j \in \Phi \\ j \lhd i}}[\Delta_i:L_j][\Rest (L_j^{R_{\mathfrak{B}}})]. \nonumber
\end{align}
In particular, given any $X$ in $\Mod{R_{\mathfrak{B}}}$, the composition factors of $\Rest(X)$ only depend on the equivalence class of the quasihereditary algebra $(A,\Phi,\unlhd)$ (namely, on $[\Delta_i:L_j]$, $[\nabla_i:L_j]$ and $\dim(\Hom{A}{\Delta_i}{\Delta_j})$ for $i,j\in \Phi$) and on the composition factors of $X$ as an $R_{\mathfrak{B}}$-module.
\end{cor}
\begin{proof}
The formula for $[\Rest(L_i^{R_{\mathfrak{B}}})]$ follows directly from the identities in Proposition \ref{prop:dimvectors} and from Lemma \ref{lem:julianvanessa}. Given any minimal element $i$ in the poset $(\Phi,\unlhd)$, then $\Rest (L_i^{R_{\mathfrak{B}}})\cong L_i^B$. By employing upwards induction on the poset $(\Phi,\unlhd)$, one is able to compute all the composition factors of $\Rest (L_i^{R_{\mathfrak{B}}})$ for every $i\in \Phi$, and these only depend on the multiplicities $[\Delta_k:L_l]$, $[\nabla_k:L_l]$ and $\dim(\Hom{A}{\Delta_k}{\Delta_l})$ for $k,l\in \Phi$. 

Observe now that $[\Rest(X)]=\sum_{i\in \Phi}[X:L_i^{R_{\mathfrak{B}}}][ \Rest(L_i^{R_{\mathfrak{B}}})]$ for every $R_{\mathfrak{B}}$-module $X$. As a result, it is also possible to calculate the composition factors of $\Rest(X)$, once the composition factors of the restriction of every simple $R_{\mathfrak{B}}$-module have been determined.
\end{proof}

The next proposition shows that the formula in \eqref{eq:dimvectorsformula} can be slightly simplified. Before proving this result, we need to clarify what is meant by immediate predecessor of an element in a poset.
\begin{defi}
An element $j$ in a poset $(\Phi,\unlhd)$ is an \emph{immediate predecessor} of $i\in \Phi$ if $j$ is a maximal element for which the strict inequality $j\lhd i$ holds. The set of all immediate predecessors of $i\in \Phi$ is denoted by $\Phi_{i^-}$.
\end{defi}
\begin{prop}
\label{prop:extra}
Let $(A,\Phi,\unlhd)$ be a quasihereditary algebra. Let $\mathfrak{B}=(B,W,\mu, \varepsilon)$ be a regular directed bocs such that $R_{\mathfrak{B}}$ is equivalent to $(A,\Phi,\unlhd)$.  For every $i \in \Phi$, the $B$-module $\Rest(L_i^{R_{\mathfrak{B}}})$ has simple socle $L_i^B$ and all its other composition factors are of the form $L_j^B$ with $j\lhd i$ and $j$ not an immediate predecessor of $i$. More precisely, 
\begin{align*}
[\Rest(L_i^{R_{\mathfrak{B}}})]=&[L_i^B]+\sum_{\substack{j\in \Phi,\, k \in \Phi\setminus \Phi_{i^-}\\ k\unlhd j \lhd i}}[\nabla_j:L_k] \dim\big(\Hom{A}{\Delta_j}{\Delta_i}\big)[\Rest(L_k^{R_{\mathfrak{B}}})]  \\
&- \sum_{\substack{j\in\Phi\setminus \Phi_{i^-} \\ j \lhd i}}[\Delta_i:L_j][\Rest (L_j^{R_{\mathfrak{B}}})].
\end{align*} 
In particular, if $i$ is either a minimal element in $(\Phi,\unlhd)$ or if all immediate predecessors of $i$ are minimal elements in $(\Phi,\unlhd)$, then $\Rest(L_i^{R_{\mathfrak{B}}})\cong L_i^B$.
\end{prop}
\begin{proof}
By Proposition \ref{prop:dimvectors}, $\Rest(L_i^{R_{\mathfrak{B}}})$ has simple socle $L_i^B$ and all its other composition factors must be of the form $L_j^B$ with $j\lhd i$. According to Corollary \ref{cor:recursivedimvectors}, $[\Rest(L_i^{R_{\mathfrak{B}}})]=[L_i^B]+ S_1 + S_2$, with 
\begin{align*}
S_1=&\sum_{\substack{j\in \Phi,\, k \in \Phi\setminus \Phi_{i^-}\\ k\unlhd j \lhd i}}[\nabla_j:L_k] \dim\big(\Hom{A}{\Delta_j}{\Delta_i}\big)[\Rest(L_k^{R_{\mathfrak{B}}})]  \\
&- \sum_{\substack{j\in\Phi\setminus \Phi_{i^-} \\ j \lhd i}}[\Delta_i:L_j][\Rest (L_j^{R_{\mathfrak{B}}})], \end{align*}\[S_2=\sum_{j\in \Phi_{i^-}}\dim\big(\Hom{A}{\Delta_j}{\Rad{\Delta_i}}\big)[\Rest(L_j^{R_{\mathfrak{B}}})]  - \sum_{j\in \Phi_{i^-}}[\Rad{\Delta_i}:L_j][\Rest (L_j^{R_{\mathfrak{B}}})].
\]
The expression $S_1$ does not give rise to any composition factors of $\Rest(L_i^{R_{\mathfrak{B}}})$ of the form $L_j^B$ with $j\in\Phi_{i^-}$; in fact, all composition factors of $\Rest(L_i^{R_{\mathfrak{B}}})$ of this form (if any) must come from $S_2$. We claim that $S_2=0$, therefore proving the equality in the statement of the proposition and showing that $L_{j}^B$ is never a composition factor of $\Rest(L_i^{R_{\mathfrak{B}}})$ when $j$ is an immediate predecessor of $i$. In order to see this, consider $j \in \Phi_{i^-}$ and let $U_j$ be the kernel of the canonical epic $\nu_j: P_j \twoheadrightarrow \Delta_j$. Note that all the simples in the top of $U_j$ are of the form $L_k$, where $k\rhd j$ (recall (2) and (3) in Definition \ref{defi:qh}). As a consequence, the only morphism in $\Hom{A}{U_j}{\Rad{\Delta_i}}$ is the zero morphism, otherwise some simple $L_k$ in the top of $U_j$ would be mapped to a composition factor of $\Rad{\Delta_i}$, leading to the contradictory inequality $j\lhd k \lhd i$. So $\Hom{A}{U_j}{\Rad{\Delta_i}}=0$ and the monic $\Hom{A}{\nu_j}{\Rad{\Delta_i}}$ is an isomorphism. It then follows that
\[\dim\big(\Hom{A}{\Delta_j}{\Rad{\Delta_i}}\big) =\dim\big(\Hom{A}{P_j}{\Rad{\Delta_i}} )=[\Rad{\Delta_i}:L_j],\]
hence $S_2=0$.
\end{proof}
The next two results summarise and clarify most of the information gathered in this subsection.
\begin{thm}
\label{thm:algo1}
Let $(A,\Phi,\unlhd)$ be a quasihereditary algebra and consider the sequence of elements $(v_{i})_{i\in\Phi}$ in the free $\Z$-module ${\Z}^{\Phi}$ on $\Phi$ defined recursively through the identity
\begin{align}
\label{eq:recursivevector}
v_{i}&={\epsilon}_i+\sum_{\substack{j,k\in \Phi\\ k\unlhd j \lhd i}}[\nabla_j:L_k] \dim\left( \Hom{A}{\Delta_j}{\Delta_i}\right) v_{k}  - \sum_{\substack{j\in\Phi \\ j \lhd i}}[\Delta_i:L_j]v_{j} \nonumber
\\&={\epsilon}_i+\sum_{\substack{j\in \Phi,\, k \in \Phi\setminus \Phi_{i^-}\\ k\unlhd j \lhd i}}[\nabla_j:L_k] \dim\left( \Hom{A}{\Delta_j}{\Delta_i}\right) v_{k}  - \sum_{\substack{j\in\Phi\setminus \Phi_{i^-} \\ j \lhd i}}[\Delta_i:L_j]v_{j},
\end{align}
where $\{{\epsilon}_i \mid i\in\Phi\}$ constitutes the standard basis of ${\Z}^{\Phi}$. Let $v_{ij}$ be the $j$th coordinate of $v_i$ for $i,j\in \Phi$. Then $(v_{ij})_{i,j\in\Phi}$ is a family of nonnegative integers satisfying $v_{ii}=1$ and $v_{ij}=0$ for every $i\in\Phi$ and $j\in\Phi_{i^-}\cup\{k\in \Phi \mid k \ntrianglelefteq i\}$. Up to relabelling of indices, the family $(v_{ij})_{i,j\in\Phi}$ is an invariant of the equivalence class $[(A,\Phi,\unlhd)]$; more precisely, it only depends on the composition factors of the standard and costandard $A$-modules and on the dimension of the $\HOM$-spaces between standard modules. Furthermore, the following assertions hold for every regular directed bocs $\mathfrak{B}=(B,W,\mu, \varepsilon)$ whose right algebra is equivalent to $(A,\Phi,\unlhd)$: 
\begin{enumerate}
\item $v_{ij}=[\Rest(L_i^{R_{\mathfrak{B}}}):L_j^B]$;
\item for every $i \in \Phi$, the $B$-module $\Rest(L_i^{R_{\mathfrak{B}}})$ has simple socle $L_i^B$ and all its other composition factors are of the form $L_j^B$ with $j\lhd i$ and $j$ not an immediate predecessor of $i$;
\item if $X$ is an $R_{\mathfrak{B}}$-module, then
\[[\Rest(X)]=\sum_{i,j\in \Phi}v_{ij}[X:L_i^{R_{\mathfrak{B}}}][L_j^B]=\sum_{\substack{i\in \Phi,\, j \in \Phi\setminus \Phi_{i^-}\\ j\unlhd i}}v_{ij}[X:L_i^{R_{\mathfrak{B}}}][L_j^B].
\] 
\end{enumerate}
\end{thm}
\begin{proof}
This is a consequence of Corollary \ref{cor:recursivedimvectors}, Proposition \ref{prop:extra} and Theorem \ref{thm:kkorephrased}. Theorem \ref{thm:kkorephrased} is needed to guarantee the existence of a regular directed bocs with right algebra equivalent to $(A,\Phi,\unlhd)$ and to consequently assure that all elements of the family $(v_{ij})_{i,j\in\Phi}$ are nonnegative; this is the quickest way to proof the nonnegativity of $(v_{ij})_{i,j\in\Phi}$.
\end{proof}

\begin{rem}
	\label{rem:matrix}
	Given a quasihereditary algebra $(A,\Phi,\unlhd)$, it may be convenient to record the nonnegative integers $v_{ij}$ described in Theorem \ref{thm:algo1} as entries of a matrix. By choosing some refinement of $(\Phi,\unlhd)$ to a total order, the integers $v_{ij}$ may be naturally arranged into a lower triangular matrix $V_{[(A,\Phi,\unlhd)]}$ with ones in the diagonal and zeros in the lower diagonal. Such a matrix $V_{[(A,\Phi,\unlhd)]}$ is invertible and it is uniquely determined by the equivalence class of $(A,\Phi,\unlhd)$, up to certain simultaneous permutations of the rows and columns. 
\end{rem}

\begin{rem}
	By using Proposition \ref{prop:dimvectors}, one concludes that the algorithm in Theorem \ref{thm:algo1} may be adapted and generalised to derive information about directed bocses.
\end{rem}

Assume that $(A,\Phi,\unlhd)$ is a quasihereditary algebra and that $\mathfrak{B}=(B,W,\mu, \varepsilon)$ is some regular directed bocs whose right algebra is equivalent to $(A,\Phi,\unlhd)$. Theorem \ref{thm:algo1} implies that the $B$-module $\Rest (L_i^{R_{\mathfrak{B}}})$ has length $\ell(\Rest (L_i^{R_{\mathfrak{B}}}))=\sum_{j\in\Phi}v_{ij}$. By adjusting the formula \eqref{eq:recursivevector}, one easily obtains a recursive method to compute the length of $\Rest (L_i^{R_{\mathfrak{B}}})$ for every $i\in\Phi$.
\begin{cor}
\label{cor:algo1}
Let $(A,\Phi,\unlhd)$ be a quasihereditary algebra and consider the sequence of integers $(l_{i})_{i\in\Phi}$ defined recursively through the identity
	\begin{align}
\label{eq:multiplicities}
l_i=&1+ \sum_{\substack{j,k \in \Phi \\ k\unlhd j \lhd i}}  l_k  [\nabla_j: L_k] \dim\left( \Hom{A}{\Delta_j}{\Delta_i}\right)  - \sum_{\substack{j\in \Phi \\  j \lhd i}} l_j  [\Delta_i :L_j]\nonumber\\
=&1+ \sum_{\substack{j\in \Phi,\, k \in \Phi\setminus \Phi_{i^-}\\ k\unlhd j \lhd i}}  l_k  [\nabla_j: L_k] \dim\left( \Hom{A}{\Delta_j}{\Delta_i}\right) - \sum_{\substack{j\in\Phi\setminus \Phi_{i^-} \\ j \lhd i}} l_j  [\Delta_i :L_j]. 
\end{align}
The elements of the sequence $(l_{i})_{i\in\Phi}$ are all positive integers and 
\begin{equation}
\label{eq:formulali}
l_i=\sum_{j\in\Phi}v_{ij}= \sum_{\substack{j\in \Phi \setminus\Phi_{i^-}\\  j \unlhd i}} v_{ij},
\end{equation}
with $v_{ij}$ as described in Theorem \ref{thm:algo1}. If $\mathfrak{B}=(B,W,\mu, \varepsilon)$ is a regular directed bocs whose right algebra is equivalent to $(A,\Phi,\unlhd)$, then $\ell ( \Rest (L_i^{R_{\mathfrak{B}}})) =l_i$ for every $i\in \Phi$.
\end{cor}
\begin{rem}
	\label{rem:matrix2}
	Let $(A,\Phi,\unlhd)$ be a quasihereditary algebra and regard the family of integers $(v_{ij})_{i,j\in \Phi}$ described in Theorem \ref{thm:algo1} as a matrix $V_{[(A,\Phi,\unlhd)]}$. For any regular directed bocs $\mathfrak{B}$, the length $l_i$ of $\Rest(L_i^{R_{\mathfrak{B}}})$ coincides with the sum of the elements in the $i$th row of $V_{[(A,\Phi,\unlhd)]}$.
\end{rem}
\begin{rem}\label{rem:bgg}
	According to the Bernstein--Gelfand--Gelfand Reciprocity Law (\cite[Lemmas 2.4, 2.5]{MR1211481}), the identities
	\[
	(X:\Delta_i)=\dim\left( \Hom{A}{X}{\nabla_i}\right) ,\quad (Y:\nabla_i)=\dim\left( \Hom{A}{\Delta_i}{Y}\right) 
	\]
	hold for every $X$ in $\mathcal{F}(\Delta)$, $Y$ in $\mathcal{F}(\nabla)$ and $i\in \Phi$. In particular, $(P_j:\Delta_i)=[\nabla_i:L_j]$ and $(Q_j:\nabla_i)=[\Delta_i:L_j]$, so the recursive formulas \eqref{eq:recursivevector} and \eqref{eq:multiplicities} may be expressed in alternative ways.
\end{rem}
We illustrate the usefulness of the previous results with two examples. The next example shows, in particular, that the sequence $(l_i)_{i\in\Phi}$ in Corollary \ref{cor:algo1} may not be monotonically increasing, even when taking the minimal adapted order of the quasihereditary algebra (recall Remark \ref{rem:minimaladapted}).
\begin{ex}
\label{ex:first}
Consider the path algebra $A=KQ$, with 
\[
Q=
\begin{tikzcd}
\overset{1}{\circ} \ar[r] & \overset{2}{\circ} & \overset{3}{\circ} \ar[r]  \ar[l]  & \overset{4}{\circ}
\end{tikzcd}.
\]
The algebra $A$ is quasihereditary with respect to the natural order on $\underline{4}=\{1,2,3,4\}$. The projective indecomposable $A$-modules are given by
\[P_1=\begin{tikzcd}[row sep=tiny, column sep=tiny]
\circled{1}\ar[d,dash]
\\ 2
\end{tikzcd}, \quad
P_2=\begin{tikzcd}[row sep=tiny, column sep=tiny]
\circled{2}
\end{tikzcd}, \quad
P_3=\begin{tikzcd}[row sep=tiny, column sep=tiny]
&\circled{3} \ar[dr,dash]\ar[dl,dash]&
\\ \circled{2} & & 4
\end{tikzcd}, \quad
P_4=\begin{tikzcd}[row sep=tiny, column sep=tiny]
\circled{4}
\end{tikzcd},
\]
and the standard modules are marked by rectangles. Using \eqref{eq:recursivevector} and Remark \ref{rem:bgg}, we obtain $v_1={\epsilon}_1$, $v_2={\epsilon}_2$,
\begin{align*}
v_3=& {\epsilon}_3+ (P_1:\Delta_1)\dim\left( \Hom{A}{\Delta_1}{\Delta_3}\right)v_1\\
&+ (P_1:\Delta_2)\dim\left( \Hom{A}{\Delta_2}{\Delta_3}\right)v_1-[\Delta_3:L_1]v_1\\
=&{\epsilon}_1+{\epsilon}_3, \\
v_4=&{\epsilon}_4,
\end{align*}
where the last equality is due to the fact that $\Delta_4$ is simple. Assume that $\mathfrak{B}=(B,W,\mu, \varepsilon)$ is a regular directed bocs whose right algebra is equivalent to $(A,\underline{4},\leq)$. According to Theorem \ref{thm:algo1}, $\Rest(L_i^{R_{\mathfrak{B}}})=L_i^B$ for $i \in\{1,2,4\}$ and the $B$-module $\Rest(L_3^{R_{\mathfrak{B}}})$ has simple socle $L_3^B$ and simple top $L_1^B$. By arranging the coordinates of $(v_i)_{i\in\underline{4}}$ into a matrix, we get
\[
V_{[(A,\underline{4},\leq)]}=\begin{pmatrix}
1 & 0 & 0 & 0 \\
0 & 1 & 0 & 0 \\
1 & 0 & 1 & 0 \\
0 & 0 & 0 & 1
\end{pmatrix}.
\]
From Remark \ref{rem:matrix2}, we derive $l_1=l_2=l_4=1$ and $l_3=2$, hence $(l_{i})_{i\in\underline{4}}$ is not increasing.
\end{ex}

\begin{ex}
	\label{ex:identity}
	Let $A$ be quasihereditary with respect to a some poset $(\Phi,\unlhd)$ of height two. By Theorem \ref{thm:algo1}, we have $v_{ij}=\delta_{ij}$ for every $i,j\in \Phi$. That is, the corresponding matrix $V_{[(A,\Phi,\unlhd)]}$ is the identity matrix and $l_i=1$ for every $i\in\Phi$.
\end{ex}
\subsection{Determining all the good quasihereditary algebras}
We wish to determine all the `good' quasihereditary algebras, i.e.~we want to find all the quasihereditary that have a regular exact Borel subalgebra and to identify, among those, the ones that contain a basic regular exact Borel subalgebra. With this goal in mind, we start with an easy consequence of the results in the previous subsection.

\begin{prop}
	\label{prop:main0}
	Let $(A,\Phi,\unlhd)$ be a quasihereditary algebra and consider the associated family of integers $(v_{ij})_{i,j \in \Phi}$ described in Theorem \ref{thm:algo1}. Let $\mathfrak{B}=(B,W,\mu, \varepsilon)$ be a regular directed bocs whose right algebra is equivalent to $(A,\Phi,\unlhd)$ and set 
	
	\[m_i=\sum\limits_{j \in \Phi}v_{ij} \dim L_j^B= \sum\limits_{\substack{j \in \Phi\setminus \Phi_{i^-} \\ j\unlhd i}}v_{ij} \dim L_j^B.\]
	The following statements hold for every $i \in \Phi$:
	\begin{enumerate}
		\item $\dim L_i^{R_{\mathfrak{B}}}=\dim ( \Rest(L_i^{R_{\mathfrak{B}}}))=m_i$;
		\item the algebra $\End{A}{\bigoplus_{i \in \Phi} P_i^{m_i}}\op$ is naturally quasihereditary with respect to the poset $(\Phi,\unlhd)$ and, further, it is equivalent to $(A,\Phi,\unlhd)$ and isomorphic to $R_{\mathfrak{B}}$; 
		\item up to isomorphism of algebras, $R_{\mathfrak{B}}$ only depends on the dimension of the simple $B$-modules, on the composition factors of the standard and costandard $A$-modules and on the dimension of the $\HOM$-spaces between standard modules.
	\end{enumerate}		
\end{prop}
\begin{proof}
	Let $\mathfrak{B}=(B,W,\mu, \varepsilon)$ be some regular directed bocs such that $(A,\Phi,\unlhd)$ is equivalent to the quasihereditary algebra $R_{\mathfrak{B}}$. It is clear that any restriction functor preserves dimensions, hence $\dim L_i^{R_{\mathfrak{B}}}=\dim ( \Rest(L_i^{R_{\mathfrak{B}}}))$. Theorem \ref{thm:algo1} assures that $[\Rest(L_i^{R_{\mathfrak{B}}}):L_j^B]=v_{ij}$, so the dimension of $\Rest(L_i^{R_{\mathfrak{B}}})$ must coincide with $m_i$. This proves (1).

	Dlab and Ringel's Standardisation Theorem (\cite[Theorem 2]{MR1211481}) implies that $R_{\mathfrak{B}}$ is Morita equivalent to $A$, so $R_{\mathfrak{B}}$ must be isomorphic to $\End{A}{\bigoplus_{i\in \Phi}P_i^{m_i}}\op$, since $m_i=\dim L_i^{R_{\mathfrak{B}}}$. 
	
	According to Theorem \ref{thm:algo1}, the integers $v_{ij}$ appearing in the definition of $m_i$ only depend on the composition factors of the standard and costandard $A$-modules and on the dimension of the $\HOM$-spaces between standard modules. This proves (3) in the statement of the proposition.
\end{proof}

We shall now translate our conclusions about regular directed bocses into statements about regular exact Borel subalgebras of quasihereditary algebras. For this, we will make use of the following result of Brzezi\'{n}ski, Koenig and K\"{u}lshammer.

\begin{thm}[{\cite[part of Theorem 3.13 and its proof]{doi:10.1112/blms.12331}, \cite[Theorem 3.6]{manuela}}]
	\label{thm:bkk}
There is a one-to-one correspondence between regular directed bocses and quasihereditary algebras containing a regular exact Borel subalgebra, which restricts to a bijection between minimal regular directed bocses and quasihereditary algebras containing a basic regular exact Borel subalgebra. This correspondence maps a regular directed bocs $\mathfrak{B}=(B,W,\mu, \varepsilon)$ to the associated embedding $\iota_{\mathfrak{B}}:B \hookrightarrow R_{\mathfrak{B}}$ of the regular exact Borel subalgebra $B$ into the quasihereditary algebra $R_{\mathfrak{B}}$.
\end{thm}

The next theorem provides a description of all quasihereditary algebras with a regular exact Borel subalgebra that lie in a given equivalence class $[(A,\Phi, \unlhd)]$. Theorem \ref{thm:main} is also a stepping stone towards the characterisation of the quasihereditary algebras possessing a regular exact Borel subalgebra in Theorem \ref{thm:mainsuper}.
\begin{thm}
\label{thm:main}
Let $(A,\Phi,\unlhd)$ be a quasihereditary algebra and consider the corresponding family of integers $(v_{ij})_{i,j \in \Phi}$ described in Theorem \ref{thm:algo1}. The following assertions hold:
\begin{enumerate}
\item for every sequence of positive integers $(k_i)_{i\in \Phi}$ there exists a regular directed bocs $\mathfrak{B}=(B,W,\mu, \varepsilon)$ such that $R_{\mathfrak{B}}$ is equivalent to $(A,\Phi,\unlhd)$ and $\dim L_i^B = k_i$ for every $i \in \Phi$;
\item for every sequence of positive integers $(k_i)_{i\in \Phi}$ there exists a quasihereditary algebra $(R,\Phi,\unlhd)\in[(A,\Phi,\unlhd)]$ which contains a regular exact Borel subalgebra $B$ satisfying $\dim L_i^B = k_i$ for every $i \in \Phi$;
\item if $(R,\Omega,\preceq)\in[(A,\Phi,\unlhd)]$ has a regular exact Borel subalgebra $B$, then it is possible to relabel the simples over $R$ by elements of $\Phi$ (concretely, consider any relabelling where $\Delta_i$ gets mapped to $\Delta_i^R$, $i\in\Phi$, under an equivalence of categories $\mathcal{F}(\Delta)\rightarrow \mathcal{F}(\Delta^R)$), so that $(R,\Omega,\preceq)$ is equivalent and isomorphic to $(\End{A}{\bigoplus_{i\in \Phi}P_i^{m_i}}\op,\Phi,\unlhd)$ with
\[m_i=\sum\limits_{\substack{j \in \Phi\setminus \Phi_{i^-} \\ j\unlhd i}}v_{ij}\dim L_i^B=\dim L_i^R.\]
\end{enumerate}
\end{thm}
\begin{proof}
	Fix a quasihereditary algebra $(A,\Phi,\unlhd)$ and a sequence of positive integers $(k_i)_{i\in \Phi}$. By Theorem \ref{thm:kkorephrased}, there exists a regular directed bocs $\mathfrak{B}=(B,W,\mu, \varepsilon)$ whose right algebra $R_{\mathfrak{B}}$ is equivalent to $(A,\Phi,\unlhd)$. Furthermore, $B$ is a regular exact Borel subalgebra of $R_{\mathfrak{B}}$. According to Lemma \ref{lem:lemma1}, the unit $\eta_X$ of the adjunction $R_{\mathfrak{B}}\otimes_B - \dashv \Rest$ is a monic for every $X$ in $\MOD{B}$. Lemma $4.5.13$ in \cite{riehl2017category} consequently implies that the functor $R_{\mathfrak{B}}\otimes_B -: \MOD{B} \rightarrow\MOD{R_{\mathfrak{B}}}$ is faithful, hence it gives rise to a monomorphism of algebras
	\[
	 \iota: B'=\End{B}{G}\op\lhook\joinrel\longrightarrow \End{R_{\mathfrak{B}}}{R_{\mathfrak{B}}\otimes_B G }\op=R',
	\]
	where $G$ is the projective genetator $\bigoplus_{i\in\Phi}\left( P_i^B\right)^{k_i}$ of $\MOD{B}$. Since $\Rest$ is an exact functor, $R_{\mathfrak{B}}\otimes_B-$ preserves projective objects (\cite[Proposition $2.3.10$]{wiebel}). Hence the algebras $B'$ and $R'$ are Morita equivalent to $B$ and $R_{\mathfrak{B}}$, respectively, and they are both quasihereditary with respect to the indexing poset $(\Phi,\unlhd)$. More precisely, $(B',\Phi,\unlhd)\in [(B,\Phi,\unlhd)]$ and $(R',\Phi,\unlhd)\in [(R_{\mathfrak{B}},\Phi,\unlhd)]=[(A,\Phi,\unlhd)]$. Furthermore, notice that $\dim L_i^{B'}=k_i$ for every $i \in \Phi$. We claim that $B'$ is a regular exact Borel subalgebra of $R'$. To check this, consider the functor $F$ given by the composition 
\[	F:
	\begin{tikzcd}[ column sep = huge]
	\MOD{B'} \ar[r, "G\otimes_{B'}-"] & \MOD{B} \ar[r, "R_{\mathfrak{B}}\otimes_B-"] & \MOD{R_{\mathfrak{B}}}\ar[r, "\Hom{R_{\mathfrak{B}}}{R_{\mathfrak{B}}\otimes_B G}{-}"] & \MOD{R'}
	\end{tikzcd}.\]
	Notice that the functors $G\otimes_{B'}-$ and $\Hom{R_{\mathfrak{B}}}{R_{\mathfrak{B}}\otimes_B G}{-}$ are equivalences, so they preserve all categorial constructions. Since $B$ is a regular exact Borel subalgebra of $R_{\mathfrak{B}}$, then $F$ is an exact functor which preserves all direct sums, sends simple modules to the corresponding standard modules and induces isomorphisms
	\[
	\Ext{B'}{n}{L^{B'}_i}{L^{B'}_j} \longrightarrow \Ext{A}{n}{F(L^{B'}_i)}{F(L^{B'}_j)}
	\]
	for every $n\geq 1$ and $i,j\in\Phi$. The Eilenberg--Watts Theorem (see Theorem 1 in \cite{MR118757} and the first paragraph of its proof) implies that $F$ is naturally isomorphic to $R'\otimes_{B'}-$, so $B'$ is a regular exact Borel subalgebra of $R'$. This proves assertion (2) in the statement of the theorem. Part (1) follows from part (2) and Theorem \ref{thm:bkk}. 
	
	For part (3), note that $R$ can be identified with the right algebra of some regular directed bocs $\mathfrak{B}=(B,W,\mu, \varepsilon)$: this is a consequence of Theorem \ref{thm:bkk}. Part (3) follows then from Proposition \ref{prop:main0}.
\end{proof}

As a direct consequence of the previous theorem, one deduces the following result.

\begin{cor}
	\label{cor:mainprimitive}
	Let $(A,\Phi,\unlhd)$ be a quasihereditary algebra. Consider the sequence of positive integers $(l_i)_{i \in \Phi}$ described in Corollary \ref{cor:algo1} and set \[R_{[(A,\Phi,\unlhd)]}=\End{A}{\bigoplus_{i\in \Phi}P_i^{l_i}}\op.\]
	The algebra $R_{[(A,\Phi,\unlhd)]}$ satisfies the following properties:
	\begin{enumerate}
		\item $R_{[(A,\Phi,\unlhd)]}$ is naturally quasihereditary with respect to $(\Phi,\unlhd)$, it is equivalent to $(A,\Phi,\unlhd)$ and it only depends on the equivalence class $[(A,\Phi,\unlhd)]$ of $(A,\Phi,\unlhd)$ (concretely, on the composition factors of the standard and costandard modules and on the dimension of $\HOM$-spaces between standard modules);
		\item up to isomorphism of algebras, $R_{[(A,\Phi,\unlhd)]}$ is the unique quasihereditary algebra equivalent to $(A,\Phi,\unlhd)$ that contains a basic regular exact Borel subalgebra;
		\item if $\mathfrak{B}=(B,W,\mu, \varepsilon)$ is a minimal regular directed bocs whose right algebra is equivalent to $(A,\Phi,\unlhd)$, then the right algebra $R_{\mathfrak{B}}$ of $\mathfrak{B}$ is isomorphic to $R_{[(A,\Phi,\unlhd)]}$ and consequently, up to isomorphism, $R_{\mathfrak{B}}$ does not depend on the choice of a minimal regular directed bocs $\mathfrak{B}$.
		\end{enumerate}
\end{cor}

\begin{proof}
The sequence $(l_i)_{i \in \Phi}$ described in Corollary \ref{cor:algo1} is totally determined by the composition factors of the standard and costandard $A$-modules and by the dimension of $\HOM$-spaces between standard modules. Part (1) in the statement of the theorem is therefore clear.

Part (2) follows from part (3) of Theorem \ref{thm:main} and from the formula \eqref{eq:formulali}. Part (3) is then a consequence of part (2) and Theorem \ref{thm:bkk}.\end{proof}

Fix a quasihereditary algebra $(A,\Phi,\unlhd)$ and picture the associated family of integers $(v_{ij})_{i,j \in \Phi}$ as a matrix $V_{[(A,\Phi,\unlhd)]}$. By Remark \ref{rem:matrix}, the matrix $V_{[(A,\Phi,\unlhd)]}$ is nonsingular (more precisely, it can be realised as a lower triangular with ones in the diagonal). Arrange the dimensions of the simples over $A$ into a column vector $b_{(A,\Phi,\unlhd)}=(\dim L_i)_{i\in\Phi}$. As a highlight in this paper, we prove that $(A,\Phi,\unlhd)$ has a regular exact Borel subalgebra exactly when the unique solution of the linear system of equations $V_{[(A,\Phi,\unlhd)]}x=b_{(A,\Phi,\unlhd)}$ is a vector with positive integer entries.
\begin{thm}
	\label{thm:mainsuper}
	Let $(A,\Phi,\unlhd)$ be a quasihereditary algebra. Consider the corresponding matrix $V_{[(A,\Phi,\unlhd)]}=(v_{ij})_{i,j \in \Phi}$ and the sequence $(l_i)_{i\in \Phi}$ described, respectively, in Theorem \ref{thm:algo1} and Remark \ref{rem:matrix}, and in Corollary \ref{cor:algo1}. The following assertions hold:
\begin{enumerate}
	\item the algebra $(A,\Phi,\unlhd)$ contains a regular exact Borel subalgebra if and only if
	there exists a sequence of positive integers $(k_i)_{i\in \Phi}$ satisfying
	\begin{equation}
	\label{eq:mostimpthm}
	\dim L_i= \sum\limits_{\substack{j \in \Phi\setminus \Phi_{i^-} \\ j\unlhd i}}v_{ij}k_j
	\end{equation}
	for every $i\in \Phi$, i.e.~if all the entries of the unique solution of the nonsingular linear system $V_{[(A,\Phi,\unlhd)]}x=(\dim L_i)_{i\in\Phi}$ are positive integers;
	\item if $(A,\Phi,\unlhd)$ has a regular exact Borel subalgebra $B$, then $(\dim L_i^B)_{i\in \Phi}$ is the unique solution of the linear system $V_{[(A,\Phi,\unlhd)]}x=(\dim L_i)_{i\in\Phi}$;
	\item the dimension of the simple modules over a regular exact Borel subalgebra of $(A,\Phi,\unlhd)$ is univocally determined by the dimension of the simple modules over $A$, by the composition factors of the standard and costandard $A$-modules and by the dimension of the $\HOM$-spaces between standard modules;
	\item if $B$ is a regular exact Borel subalgebra of $(A,\Phi,\unlhd)$ and $\Rest$ is the corresponding restriction functor, then $[\Rest(L_i):L_j^B]=v_{ij}$ and $\ell(\Rest(L_i))=l_i$ for every $i,j\in \Phi$;
	\item the algebra $(A,\Phi,\unlhd)$ contains a basic regular exact Borel subalgebra if and only if $\dim L_i=l_i$ for every $i\in \Phi$.
\end{enumerate}
\end{thm}

\begin{proof}
We start by proving the assertions (1), (2) and (3) in the statement of the theorem. Suppose first that $(A,\Phi,\unlhd)$ contains a regular exact Borel subalgebra $B$. Part (3) of Theorem \ref{thm:main} implies the existence of positive integers $k_i=\dim L_i^B$, $i\in\Phi$, satisfying \eqref{eq:mostimpthm}. That is, the vector $k=(k_i)_{i\in\Phi}$ is the unique solution of the nonsingular linear system $V_{[(A,\Phi,\lhd)]}x=(\dim L_i)_{i\in\Phi}$. This proves one of the implications in part (1) and also part (2). Part (3) follows from the fact that the entries $v_{ij}$ of the matrix $V_{[(A,\Phi,\lhd)]}$ only depend on the composition factors of the standard and costandard $A$-modules and on the dimension of the $\HOM$-spaces between standard modules (recall Theorem \ref{thm:algo1}). To prove the remaining implication in part (1), assume that $(A,\Phi,\unlhd)$ is such that the linear system $V_{[(A,\Phi,\lhd)]}x=(\dim L_i)_{i\in\Phi}$ has a solution $k=(k_i)_{i\in\Phi}$ satisfying $k_i\in \N$ for every $i\in \Phi$. By part (2) of Theorem \ref{thm:main}, there exists a quasihereditary algebra $(R,\Phi,\unlhd)\in[(A,\Phi,\unlhd)]$ which contains a regular exact Borel subalgebra $B$ satisfying $\dim L_i^B = k_i$ for every $i \in \Phi$. According to part (3) of Theorem \ref{thm:main}, $(R,\Phi,\unlhd)$ is equivalent and isomorphic to $(\End{A}{\bigoplus_{i\in \Phi}P_i^{m_i}}\op,\Phi,\unlhd)$ with
\[m_i=\sum\limits_{\substack{j \in \Phi\setminus \Phi_{i^-} \\ j\unlhd i}}v_{ij}\dim L_i^B=\sum\limits_{\substack{j \in \Phi\setminus \Phi_{i^-} \\ j\unlhd i}}v_{ij}k_i=\dim L_i,\]
where the last equality follows from $k$ being a solution of $V_{[(A,\Phi,\lhd)]}x=(\dim L_i)_{i\in\Phi}$. Hence $R$ is isomorphic to $A$ and $A$ contains a regular exact Borel subalgebra.

In order to prove assertion (4), note that the inclusion of a regular exact Borel subalgebra $B$ into a quasihereditary algebra $(A,\Phi,\unlhd)$ can be regarded as the inclusion of $B$ into the right algebra $R_{\mathfrak{B}}$ of a regular directed bocs $\mathfrak{B}=(B,W,\mu, \varepsilon)$, by Theorem \ref{thm:bkk}. Hence, assertion (4) follows from part (1) of Theorem \ref{thm:algo1} and Corollary \ref{cor:algo1}.

Part (5) follows from Corollary \ref{cor:mainprimitive} or, alternatively, from parts (1) and (2).
\end{proof}
\begin{ex}
	\label{ex:cannotwait}
	Let $(A,\underline{4},\leq)$ be the (basic) quasihereditary algebra discussed in Example \ref{ex:first} and notice the matrix $V_{[(A,\underline{4},\leq)]}$ computed in there. The unique solution of the linear system $V_{[(A,\underline{4},\leq)]}x=(1)_{i\in \underline{4}}$ is $(k_i)_{i\in \underline{4}}$ with $k_1=k_2=k_4=1$ and $k_3=0$. By Theorem \ref{thm:mainsuper}, $(A,\underline{4},\leq)$ has no regular exact Borel subalgebra. 
	\end{ex}
\begin{ex}
	Recall the conclusions in Example \ref{ex:identity}. Part (1) of Theorem \ref{thm:mainsuper} consequently implies that every quasihereditary algebra with an indexing poset of height at most $2$ has a regular exact Borel subalgebra. In particular, every quasihereditary algebra with at most two nonisomorphic simple modules has an exact Borel subalgebra: this had already been proved in \cite[Theorem $4.58$]{kulshammer2016bocs}.  
\end{ex}

The previous results prompt us to reserve a special name for the quasihereditary algebras that contain a regular exact Borel subalgebra.

\begin{defi}A quasihereditary algebra $(A,\Phi,\unlhd)$ is \emph{good} if it contains a regular exact Borel subalgebra. A good quasihereditary algebra is \emph{minimal} if it contains a basic regular exact Borel subalgebra.
\end{defi}
	
Theorem \ref{thm:mainsuper} gives a criterion to single out the good quasihereditary algebras and Theorem \ref{thm:main} describes all the good algebras belonging to a fixed equivalence class $[(A,\Phi,\unlhd)]$. According to Corollary \ref{cor:mainprimitive}, there exists essentially one minimal good quasihereditary per equivalence class: the algebra $(R_{[(A,\Phi,\unlhd)]}, \Phi,\unlhd)$ is the minimal good representative of $[(A,\Phi,\unlhd)]$.

We conclude this subsection with an application of Corollary \ref{cor:mainprimitive}.

\begin{ex}
Let $(A,\underline{4},\leq)$ be the (basic) quasihereditary algebra discussed in Example \ref{ex:first}. We have checked in Example \ref{ex:cannotwait} that $(A,\underline{4},\leq)$ contains no regular exact Borel subalgebra. Recall that $l_1=1$, $l_2=1$, $l_3=2$ and $l_4=1$. According to Corollary \ref{cor:mainprimitive}, the algebra $(R_{[(A,\underline{4},\leq)]}, \underline{4},\leq)$ with $R_{[(A,\underline{4},\leq)]}=\End{A}{P_1\oplus P_2\oplus P_3^2\oplus P_4}\op$ is a minimal good quasihereditary algebra. To be precise, up to isomorphism of algebras, $R_{[(A,\underline{4},\leq)]}$ is the unique quasihereditary algebra in $[(A,\underline{4},\leq)]$ that contains a basic regular exact Borel subalgebra, say $B$. Since $B$ is regular (recall Definiton \ref{defi:borelprops}), then
\[\Ext{B}{1}{L_i^B}{L_j^B} \cong \Ext{R_{[(A,\underline{4},\leq)]}}{1}{\Delta_i^{R_{[(A,\underline{4},\leq)]}}}{\Delta_j^{R_{[(A,\underline{4},\leq)]}}} \cong \Ext{A}{1}{\Delta_i}{\Delta_j}.\]

Using this information, it is not difficult to conclude that $B$ must be isomorphic to the path algebra of the quiver
\[
\begin{tikzcd}
\overset{1}{\circ} \ar[r, "a"]\ar[rr, "b"' ,bend right] & \overset{2}{\circ} & \overset{3}{\circ}   \ar[r, "c"]  & \overset{4}{\circ}
\end{tikzcd}.
\]
The path $cb$ is nonzero in $B$ because there exists an indecomposable $A$-module which is filtered by $\Delta_1$, $\Delta_3$ and $\Delta_4$. 

Call to mind the structure of the projective indecomposable $A$-modules, described in Example \ref{ex:first}. The elements of $R_{[(A,\underline{4},\leq)]}=\End{A}{P_1\oplus P_2\oplus P_3^2\oplus P_4}\op$ can be regarded as $5\times 5$ matrices whose entries are homomorphisms between projective indecomposable $A$-modules. Consider the assignment 
\begin{gather*}
e_1\longmapsto\begin{pmatrix}
1_{P_1} & 0 & 0 & 0 & 0 \\
0 & 0 & 0 & 0 & 0 \\
0 & 0 & 0 & 0 & 0 \\
0 & 0 & 0 & 1_{P_3} & 0 \\
0 & 0 & 0 & 0 & 0 
\end{pmatrix},\quad e_2\longmapsto
\begin{pmatrix}
0 & 0 & 0 & 0 & 0 \\
0 & 1_{P_2} & 0 & 0 & 0 \\
0 & 0 & 0 & 0 & 0 \\
0 & 0 & 0 & 0 & 0 \\
0 & 0 & 0 & 0 & 0 
\end{pmatrix},\\ e_3\longmapsto\begin{pmatrix}
0 & 0 & 0 & 0 & 0 \\
0 & 0 & 0 & 0 & 0 \\
0 & 0 & 1_{P_3} & 0 & 0 \\
0 & 0 & 0 & 0 & 0 \\
0 & 0 & 0 & 0 & 0 
\end{pmatrix},\quad e_4 \longmapsto\begin{pmatrix}
0 & 0 & 0 & 0 & 0 \\
0 & 0 & 0 & 0 & 0 \\
0 & 0 & 0 & 0 & 0 \\
0 & 0 & 0 & 0 & 0 \\
0 & 0 & 0 & 0 & 1_{P_4}
\end{pmatrix},
\end{gather*}\begin{gather*}a\longmapsto
\begin{pmatrix}
0 & f & 0 & 0 & 0 \\
0 & 0 & 0 & 0 & 0 \\
0 & 0 & 0 & 0 & 0 \\
0 & g & 0 & 0 & 0 \\
0 & 0 & 0 & 0 & 0 
\end{pmatrix}
 ,\quad b\longmapsto
\begin{pmatrix}
0 & 0 & 0 & 0 & 0 \\
0 & 0 & 0 & 0 & 0 \\
0 & 0 & 0 & 0 & 0 \\
0 & 0 & 1_{P_3} & 0 & 0 \\
0 & 0 & 0 & 0 & 0
\end{pmatrix},\\
c\longmapsto
\begin{pmatrix}
0 & 0 & 0 & 0 & 0 \\
0 & 0 & 0 & 0 & 0 \\
0 & 0 & 0 & 0 & h \\
0 & 0 & 0 & 0 & 0 \\
0 & 0 & 0 & 0 & 0 
\end{pmatrix},\end{gather*}
where the maps $f\in\Hom{A}{P_2}{P_1}$, $g\in\Hom{A}{P_2}{P_3}$ and $h\in\Hom{A}{P_4}{P_3}$ embed the simple projectives $P_2=L_2$ and $P_4=L_4$ into the socle of $P_1$ and $P_3$, respectively. One can check that this correspondence gives rise to an injective homomorphism of algebras $\iota:B \hookrightarrow R_{[(A,\underline{4},\leq)]}$, where the multiplication in $R_{[(A,\underline{4},\leq)]}$ is identified with the operation opposite to the usual matrix multiplication. The embedding $\iota$ turns $B$ into a (basic) regular exact Borel subalgebra of $(R_{[(A,\underline{4},\leq)]},\underline{4},\leq)$.
\end{ex}
\subsection{More on regular exact Borel subalgebras and regular directed bocses}

We shall now see that the Cartan matrix of a regular exact Borel subalgebra of any quasihereditary algebra in $[(A,\Phi,\unlhd)]$ only depends on the equivalence class $[(A,\Phi, \unlhd)]$, namely on the composition factors of the standard and costandard $A$-modules and on the dimension of the $\HOM$-spaces between standard modules. 

The \emph{Cartan matrix} $C_{[A]}$ of an algebra $A$ whose simple modules are indexed by a set $\Phi$ is given by $([P_i:L_j])_{i,j\in\Phi}$. Note that $C_{[A]}$ is Morita invariant and observe that $C_{[A\op]}=(C_{[A]})^T$. 

If $A$ is quasihereditary with respect to a poset $(\Phi,\unlhd)$, define the \emph{$\Delta$-decomposition matrix} and the \emph{$\nabla$-decomposition matrix}, respectively, as
\[
D_{[(A,\Phi, \unlhd)]}^{\Delta}=\left(\left[\Delta_i:L_j \right]  \right)_{i,j\in\Phi},\quad D_{[(A,\Phi, \unlhd)]}^{\nabla}=\left(\left[\nabla_i:L_j \right]  \right)_{i,j\in\Phi}.
\]
By considering a refinement of $(\Phi,\unlhd)$ to a total order, $D_{[(A,\Phi, \unlhd)]}^{\Delta}$ and $D_{[(A,\Phi, \unlhd)]}^{\nabla}$ can be seen as lower triangular matrices with ones in the diagonal, so they are clearly invertible. Consider also the \emph{$\Delta$-} and the \emph{$\nabla$-filtration matrices} 
\[
F_{[(A,\Phi, \unlhd)]}^{\Delta}=\left(\left(P_i:\Delta_j \right)  \right)_{i,j\in\Phi},\quad F_{[(A,\Phi, \unlhd)]}^{\nabla}=\left(\left(Q_i:\nabla_j \right)  \right)_{i,j\in\Phi},
\]
and note that $F_{[(A,\Phi, \unlhd)]}^{\Delta}=(D_{[(A,\Phi, \unlhd)]}^{\nabla})^T$ and $F_{[(A,\Phi, \unlhd)]}^{\nabla}=(D_{[(A,\Phi, \unlhd)]}^{\Delta})^T$ (recall Remark \ref{rem:bgg}). The decomposition and filtration matrices are invariant under equivalence of quasihereditary algebras. Notice that $C_{[B]}=F_{[(B,\Phi, \unlhd)]}^{\Delta}$ and $C_{[B\op]}=D_{[(B,\Phi, \unlhd)]}^{\nabla}$ for any quasihereditary algebra $(B,\Phi,\unlhd)$ with simple standard modules. 

In the following theorem, we use the notation $\ell_{\Delta}(X)$ for the number of standard modules appearing in a $\Delta$-filtration of $X$ in $\mathcal{F}(\Delta)$.

\begin{thm}
\label{thm:mainthm2}
Let $(A,\Phi,\unlhd)$ be a quasihereditary algebra. Consider the corresponding matrix $V_{[(A,\Phi,\unlhd)]}=(v_{ij})_{i,j \in \Phi}$ and the sequence $(l_i)_{i\in \Phi}$ described, respectively, in Theorem \ref{thm:algo1} and Remark \ref{rem:matrix}, and in Corollary \ref{cor:algo1}. Assume that $B$ is a regular exact Borel subalgebra of some quasihereditary algebra equivalent to $(A,\Phi,\unlhd)$ or suppose (equivalently) that $\mathfrak{B}=(B,W,\mu, \varepsilon)$ is a regular directed bocs whose right algebra is equivalent to $(A,\Phi,\unlhd)$. The simples over $B$ may be naturally indexed by the poset $(\Phi,\unlhd)$ so that the following holds:
\begin{enumerate}
\item $[Q_i^B:L_j^B]=[P_j^B:L_i^B]=\sum\limits_{k \in \Phi}[\nabla_i:L_k]v_{kj}=\sum\limits_{\substack{k \in \Phi \\ j\unlhd k \unlhd i}}[\nabla_i:L_k]v_{kj}$;
\item $\ell \left( Q_i^B\right) =\sum\limits_{\substack{j \in \Phi \\ j\unlhd i}}[\nabla_i:L_j]\ell_j$ and $\ell \left( P_i^B\right) =\sum\limits_{\substack{k \in \Phi  \\ i \unlhd k}}\ell_{\Delta}(P_k)v_{ki}$;
\item $C_{[B\op]}=D_{[(B,\Phi, \unlhd)]}^{\nabla}=D_{[(A,\Phi, \unlhd)]}^{\nabla}V_{[(A,\Phi, \unlhd)]}$;
\item $C_{[B]}=F_{[(B,\Phi, \unlhd)]}^{\Delta}=(V_{[(A,\Phi, \unlhd)]})^T F_{[(A,\Phi, \unlhd)]}^{\Delta}$;
\item $(\dim Q_i^B)_{i\in\Phi}=D_{[(A,\Phi, \unlhd)]}^{\nabla}V_{[(A,\Phi, \unlhd)]}(\dim L_i^B)_{i\in\Phi}$;
\item $(\dim P_i^B)_{i\in\Phi}=(V_{[(A,\Phi, \unlhd)]})^T F_{[(A,\Phi, \unlhd)]}^{\Delta}(\dim L_i^B)_{i\in\Phi}$;
\item $\dim B=\sum\limits_{i\in\Phi} \dim Q_i^B \dim L_i^B=\sum\limits_{i\in\Phi} \dim P_i^B \dim L_i^B$;
\item $\dim W=\dim B + \sum\limits_{\substack{i,j\in\Phi\\ j\lhd i}}\dfrac{\dim \Big(\Hom{R_{\mathfrak{B}}}{\Delta_j^{R_{\mathfrak{B}}}}{\Delta_i^{R_{\mathfrak{B}}}}\Big)}{\dim L_j^B}\dim P_i^B \dim Q_j^B$.
\end{enumerate}
In particular, the invariants in (1) to (4) are completely determined by the composition factors of the standard and costandard $A$-modules and by the dimension of the $\HOM$-spaces between standard modules, so they only depend on $[(A,\Phi,\unlhd)]$. The remaining invariants depend additionally on the dimension of the simple $B$-modules.
\end{thm}
\begin{proof}
According to Theorem \ref{thm:bkk}, we may restrict ourselves to the setting where $\mathfrak{B}=(B,W,\mu, \varepsilon)$ is a regular directed bocs whose right algebra is equivalent to $(A,\Phi,\unlhd)$. In this case, $B$ is a (regular) exact Borel subalgebra of $R_{\mathfrak{B}}$. Recall the adjoint triple in \eqref{eq:adointtriple}, namely that $\Rest$ is left adjoint to $\Hom{B}{R_{\mathfrak{B}}}{-}$. Since $\Rest$ is an exact functor, $\Hom{B}{R_{\mathfrak{B}}}{-}$ preserves injective objects (\cite[Proposition $2.3.10$]{wiebel}). Consequently, $\Hom{B}{R_{\mathfrak{B}}}{Q_j^B}$ is an injective $R_{\mathfrak{B}}$-module. Note that
\begin{align*}
\dim\left(\Hom{R_{\mathfrak{B}}}{L_k^{R_{\mathfrak{B}}}}{\Hom{B}{R_{\mathfrak{B}}}{Q_j^B}}\right)&=\dim\left(\Hom{B}{\Rest(L_k^{R_{\mathfrak{B}}})}{Q_j^B}\right)\\
&=[\Rest(L_k^{R_{\mathfrak{B}}}):L_j^B]=v_{kj},
\end{align*}
by Theorem \ref{thm:algo1}. Hence
\[
\Hom{B}{R_{\mathfrak{B}}}{Q_j^B}\cong \bigoplus_{k\in \Phi} (Q_k^{R_{\mathfrak{B}}})^{v_{kj}}.\]
Using Theorem \ref{thm:koenigborel}, we get
\[
\Hom{B}{Q_i^B}{Q_j^B}\cong\Hom{B}{\Rest(\nabla_i^{R_{\mathfrak{B}}})}{Q_j^B}\cong\Hom{R_{\mathfrak{B}}}{\nabla_i^{R_{\mathfrak{B}}}}{\Hom{B}{R_{\mathfrak{B}}}{Q_j^B}}.\]
It follows that
\begin{align*}
\dim\left(\Hom{B}{Q_i^B}{Q_j^B}\right)&=\dim\left(\Hom{R_{\mathfrak{B}}}{\nabla_i^{R_{\mathfrak{B}}}}{\bigoplus_{k\in \Phi} (Q_k^{R_{\mathfrak{B}}})^{v_{kj}}}\right)\\
&=\sum_{k\in \Phi}[\nabla_i^{R_{\mathfrak{B}}}:L_k^{R_{\mathfrak{B}}}] v_{kj}=\sum_{\substack{k \in \Phi \\ j\unlhd k \unlhd i}}[\nabla_i:L_k] v_{kj}.
\end{align*}
Now note that 
\[[Q_i^B:L_j^B]=\dim\big(\Hom{B}{Q_i^B}{Q_j^B}\big)=\dim\big(\Hom{B}{P_j^B}{Q_i^B}\big)=[P_j^B:L_i^B],\]
and this concludes the proof of the identities in (1). The first equality in (2) follows from (1) and from \eqref{eq:formulali}. Analogously, the second identity in (2) results from (1) and from the fact that $[\nabla_j:L_k]=(P_k:\Delta_j)$, as pointed out in Remark \ref{rem:bgg}.  The equality in (3) is merely a rephrasing of (1). Note that (4) is obtained from (3) by transposition. The equalities in (5) and (6) follow from (3) and (4), respectively, and from the definition of Cartan matrix. The identity in (7) holds for any algebra over an algebraically closed field. Finally, the equality in (8) follows from the decompositions $W\cong B\oplus \Ker{\varepsilon}$ in $\MOD{B}$ and $\Ker{\varepsilon}\cong\bigoplus_{i,j \in \Phi}(B e_i \otimes_K e_j B)^{n_{ij}}$ in $\MOD{B\otimes_K B\op}$, where
\[
n_{ij}=\frac{\dim \Big(\Hom{R_{\mathfrak{B}}}{\Delta_j^{R_{\mathfrak{B}}}}{\Rad{\Delta_i^{R_{\mathfrak{B}}}}}\Big)}{\dim L_j^B},
\]
as deduced in Lemma \ref{lem:julianvanessa}.
\end{proof}

The following example demonstrates how the results deduced so far can be used to perform concrete computations.
\begin{ex}
\label{ex:second}
Let $n\in\N$ and consider the algebra $A=KQ/I$, with
\[
Q=\begin{tikzcd}[ampersand replacement=\&]
\overset{1}{\circ} \arrow[bend left]{r}{\alpha_1}\& \overset{2}{\circ}\arrow[bend left]{l}{\beta_1} \arrow[bend left]{r}{\alpha_2}\&  \cdots\arrow[bend left]{l}{\beta_2} \arrow[bend left]{r}{\alpha_{n-2}}  \& \overset{n-1}{\circ} \arrow[bend left]{l}{\beta_{n-2}} \arrow[bend left]{r}{\alpha_{n-1}} \&
\overset{n}{\circ}\arrow[bend left]{l}{\beta_{n-1}}
\end{tikzcd}
\]
and $I$ the admissible ideal of $KQ$ generated by the relations $\alpha_{n-1}\beta_{n-1}$ and by $\alpha_{i+1}\alpha_i$, $\beta_i \beta_{i+1}$ and $\alpha_i \beta_{i} - \beta_{i+1}\alpha_{i+1}$, for $i=1 , \ldots, n-2$. The algebra $A$ is quasihereditary with respect to the natural ordering on $\underline{n}=\{1,\ldots, n\}$. In \cite{MR1206201}, Erdmann proved that the finite representation type blocks of a Schur algebra are always Morita equivalent to $A$ for particular instances of $n$.

The projective indecomposable $A$-modules may be represented in the following way
\[
P_1=\begin{tikzcd}[ampersand replacement=\&, row sep =tiny, column sep = tiny]
\circled{1} \arrow[dash]{d} \\
2 \arrow[dash]{d} \\
1
\end{tikzcd}, \quad
P_i=\begin{tikzcd}[ampersand replacement=\&, row sep =tiny, column sep = tiny]
\& \circled{$i$} \arrow[dash]{dr} \arrow[dash]{dl} \& \\
\circled{$i-1$} \arrow[dash]{dr}\& \& i+1 \arrow[dash]{dl} \\
\& i  \&
\end{tikzcd}, \,\, i=2, \ldots, n-1,\quad
P_n=\begin{tikzcd}[ampersand replacement=\&, row sep =tiny, column sep = tiny]
\circled{$n$} \arrow[dash]{d} \\
\circled{$n-1$} 
\end{tikzcd},
\]
where the standard modules are marked by rectangles. Using \eqref{eq:recursivevector} and the Bernstein--Gelfand--Gelfand Reciprocity Law (Remark \ref{rem:bgg}), we get
\[
v_i={\epsilon}_i+\sum_{\substack{1\leq k\leq i-2 \\ k\leq j\leq i-1}} (P_k:\Delta_j)\dim\big(\Hom{A}{\Delta_j}{\Delta_i}\big)v_k-\sum_{1\leq j\leq i-2}[\Delta_i:L_j]v_j.
\]
Note that most of the summands in the expression above are zero. In fact, $v_1={\epsilon}_1$, $v_2={\epsilon}_2$ and $v_{i}={\epsilon}_i +v_{i-2}$ for $2 < i \leq n$. As a result,
\[v_{i}=\sum_{\substack{1 \leq j\leq i\\j\equiv i(\modu 2)}} {\epsilon}_j.\]
Hence $v_{ij}=1$ if $1\leq j\leq i$ and $j\equiv i(\modu 2)$, and $v_{ij}=0$ otherwise. If $\mathfrak{B}=(B,W,\mu, \varepsilon)$ is a regular directed bocs whose right algebra is equivalent to $(A,\underline{n},\leq)$, then Theorem \ref{thm:algo1} implies that the $B$-module $\Rest\big(L_i^{R_\mathfrak{B}}\big)$ has composition factors $L_i^B$, $L_{i-2}^B$, $L_{i-4}^B$, and so on. 

Using \eqref{eq:formulali}, we obtain $l_i=\lceil\frac{i}{2}\rceil$, so the minimal good representative of $[(A,\underline{n},\leq)]$ is
\[
R_{[(A,\underline{n},\leq)]}=\End{A}{\bigoplus_{i=1}^n P_i^{\lceil\frac{i}{2}\rceil}}\op,
\]
that is, up to isomorphism of algebras, this is the unique quasihereditary algebra equivalent to $(A,\underline{n},\leq)$ that contains a basic regular exact Borel subalgebra (recall Corollary \ref{cor:mainprimitive}). 

Suppose now that $B$ is a regular exact Borel subalgebra of some algebra equivalent to $(A,\underline{n},\leq)$. By Theorem \ref{thm:mainthm2}, $Q_1^B\cong L_1^B$ and 
\[[Q_i^B:L_j^B]=[P_j^B:L_i^B]=(P_i:\Delta_i)v_{ij}+(P_{i-1}:\Delta_{i})v_{(i-1)j}=v_{ij}+v_{(i-1)j},\]
for $1<i\leq n$ and $j\in \underline{n}$. Consequently,  $[Q_i^B:L_j^B]=[P_j^B:L_i^B]=1$ for every $i,j\in \underline{n}$ with $j\leq i$. The Cartan matrix of $B$ is therefore upper triangular with all the entries on the diagonal and above equal to one. Moreover, $\ell (Q_i^B)=i$ and $\ell(P_i^B)=n-i+1$ for every $i\in\underline{n}$.

If, in addition, $B$ is basic, then $\dim Q_i^B=i$, $\dim P_i^B=n-i+1$ and 
\[\dim B=\sum_{i=1}^n i=\dfrac{n(n+1)}{2},\]
which agrees with the computations in \cite[§5.1]{kulshammer2016bocs}, based on the work in \cite{klamt2011ainfinity,KLAMT2012323}. In this case,
\begin{align*}\dim W&=\dim B+\sum_{\substack{i,j\in\underline{n}\\ j< i}}\dim\big(\Hom{A}{\Delta_j}{\Delta_i}\big)\dim P_i^B \dim Q_j^B\\
&=\dfrac{n(n+1)}{2}+\sum_{i=2}^n \dim P_i^B \dim Q_{i-1}^B=\dfrac{n(n+1)}{2}+\sum_{i=2}^n (n-i+1)(i-1)\\
&=\dfrac{n(n+1)}{2}+\dfrac{(n-1)n(n+1)}{6}=\dfrac{n(n+1)(n+2)}{6}.
\end{align*}
\end{ex}
We shall now deduce two further consequences of the previous results.
 
\begin{cor}
\label{cor:maincor21}
Let $(A,\Phi,\unlhd)$ be a quasihereditary algebra containing a regular exact Borel subalgebra $B$. Then the dimensions $\dim L_i^B$, $\dim P_i^B$, $\dim Q_i^B$ and $\dim B$ are completely and univocally determined by the composition factors of the standard and costandard $A$-modules, by the dimension of the $\HOM$-spaces between standard $A$-modules and by the dimension of the simple $A$-modules.
\end{cor}
\begin{proof}
	Combine part (3) of Theorem \ref{thm:mainsuper} and Theorem \ref{thm:mainthm2}.
\end{proof}

A one-to-one correspondence between the isomorphism classes of minimal regular directed bocses and the equivalence classes of quasihereditary algebras, sending a bocs to the equivalence class of its right algebra, shall be provided in upcoming work of K\"ulshammer and Miemietz (see \cite[Theorem 4.26]{kulshammer2016bocs}; we refer to \cite[§17.12]{brzezinski2003corings} for the definition of isomorphism of bocses). We have not been able to show that any two minimal regular directed bocses whose corresponding right algebras are equivalent as quasihereditary algebras must be isomorphic as bocses. However, we prove that any two such bocses have a lot of data in common.
\begin{cor}
\label{cor:maincor22}
Let $\mathfrak{B}=(B,W,\mu, \varepsilon)$ and $\mathfrak{B}'=(B',W',\mu', \varepsilon')$ be two regular directed bocses whose corresponding right algebras are equivalent. It is possible to label the simples over $B$ and $B'$ by the elements of same set $\Phi$ in a way which is compatible with the equivalence of $R_{\mathfrak{B}}$ and $R_{\mathfrak{B}'}$. For such a labelling, the Cartan matrices of $B$ and $B'$ coincide,  and consequently $\ell(P_i^B)=\ell(P_i^{B'})$ and $\ell(Q_i^B)=\ell(Q_i^{B'})$ for every $i\in\Phi$. Moreover, the following conditions are equivalent:
\begin{enumerate}
\item $\dim L_i^B=\dim L_i^{B'}$ for every $i \in \Phi$;
\item $\dim L_i^{R_{\mathfrak{B}}}=\dim L_i^{R_{\mathfrak{B'}}}$ for every $i \in \Phi$;
\item $\dim P_i^B=\dim P_i^{B'}$ for every $i \in \Phi$;
\item $\dim Q_i^B=\dim Q_i^{B'}$ for every $i \in \Phi$;
\item $\dim P_i^{R_{\mathfrak{B}}}=\dim P_i^{R_{\mathfrak{B}'}}$ for every $i \in \Phi$;
\item $\dim Q_i^{R_{\mathfrak{B}}}=\dim Q_i^{R_{\mathfrak{B}'}}$ for every $i \in \Phi$;
\item $\dim \Delta_i^{R_{\mathfrak{B}}}=\dim \Delta_i^{R_{\mathfrak{B}'}}$ for every $i \in \Phi$;
\item $\dim \nabla_i^{R_{\mathfrak{B}}}=\dim \nabla_i^{R_{\mathfrak{B}'}}$ for every $i \in \Phi$.
\end{enumerate}
 If one of the conditions (1) to (8) is satisfied, then $R_{\mathfrak{B}}$ and $R_{\mathfrak{B}'}$ are isomorphic algebras, $\dim B=\dim B'$ and $\dim W=\dim W'$. In particular, if $\mathfrak{B}$ and $\mathfrak{B'}$ are both minimal, then conditions (1) to (8) are satisfied, $R_{\mathfrak{B}}\cong R_{\mathfrak{B}'}$, $\dim B=\dim B'$ and $\dim W=\dim W$.
\end{cor}
\begin{proof}
	Let $\mathfrak{B}=(B,W,\mu, \varepsilon)$ and $\mathfrak{B}'=(B',W',\mu', \varepsilon')$ be two regular directed bocses whose right algebras are equivalent. Index the simples over $B$ and $B'$ by a poset $(\Phi,\unlhd)$ in a way which is compatible with the equivalence of $R_{\mathfrak{B}}$ and $R_{\mathfrak{B}'}$. That is, the labelling should be such that $(R_{\mathfrak{B}},\Phi,\unlhd)$ and $(R_{\mathfrak{B}'},\Phi,\unlhd)$ are equivalent via an equivalence of categories $\mathcal{F}(\Delta^{R_{\mathfrak{B}}})\rightarrow \mathcal{F}(\Delta^{R_{\mathfrak{B}'}})$ that maps $\Delta_{i}^{R_{\mathfrak{B}}}$ to $\Delta_{i}^{R_{\mathfrak{B}'}}$. 
	
	By Theorem \ref{thm:mainthm2}, the Cartan matrices of $B$ and $B'$ must coincide and the same holds for the lengths of the projective and injective indecomposables over $B$ and $B'$. This proves the initial part of the corollary.
	
	We now show that the statements (1) to (8) are equivalent. Since $V_{[(R_{\mathfrak{B}},\Phi,\unlhd)]}$ equals $V_{[(R_{\mathfrak{B}'},\Phi,\unlhd)]}$, Theorem \ref{thm:mainsuper} implies that (1) and (2) are equivalent. The equivalence of the statements (1), (3) and (4) follows from the fact that the (invertible) matrices $C_{[B]}$ and $C_{[B']}$ (and $C_{[B\op]}$ and $C_{[B'\op]}$) coincide. Similarly, the conditions (2), (7) and (8) are equivalent because $D_{[(R_{\mathfrak{B}},\Phi,\unlhd)]}^\Delta=D_{[(R_{\mathfrak{B}'},\Phi,\unlhd)]}^\Delta$ and $D_{[(R_{\mathfrak{B}},\Phi,\unlhd)]}^\nabla=D_{[(R_{\mathfrak{B}'},\Phi,\unlhd)]}^\nabla$ and these matrices are invertible. Since $F_{[(R_{\mathfrak{B}},\Phi,\unlhd)]}^\Delta$ and $F_{[(R_{\mathfrak{B}'},\Phi,\unlhd)]}^\Delta$ coincide, then the assertions (5) and (7) must also be equivalent. In an analogous way, using the identity $F_{[(R_{\mathfrak{B}},\Phi,\unlhd)]}^\nabla=F_{[(R_{\mathfrak{B}'},\Phi,\unlhd)]}^\nabla$, we conclude that (6) and (8) are equivalent statements.
	
	 Finally, if condition (2) is met, it follows from part (3) of Theorem \ref{thm:main} that the algebras $R_{\mathfrak{B}}$ and $R_{\mathfrak{B}'}$ are isomorphic. By Theorem \ref{thm:mainthm2}, it is clear that the identities $\dim B=\dim B'$ and $\dim W=\dim W$ hold whenever one of the assertions (1) to (8) is satisfied. 
\end{proof}
\begin{rem}
In the light of the result announced in \cite[Theorem 4.26]{kulshammer2016bocs}, one should expect that any two regular directed bocses $\mathfrak{B}=(B,W,\mu, \varepsilon)$ and $\mathfrak{B}'=(B',W',\mu', \varepsilon')$ whose right algebras are equivalent and for which one of the conditions (1) to (8) in Corollary \ref{cor:maincor22} holds must be isomorphic as bocses.
\end{rem}

We close this section with an application of Theorem \ref{thm:mainthm2} to two nonequivalent quasihereditary algebras having ``identical data". 
\begin{ex}
\label{ex:third}
Fix $n\in\N$. Consider the bound quiver algebras $A=KQ/J$ and $A'=KQ/J'$, where $Q$ is the quiver with $n$ vertices in Example \ref{ex:second} and the ideals $J$ and $J'$ are defined by
\begin{gather*}
J=\langle \alpha_{n-1} \beta_{n-1},\, \alpha_i \beta_{i} - \beta_{i+1}\alpha_{i+1}\text{ for } i=1 , \ldots, n-2\rangle,\\
J'=\langle \alpha_i \beta_i \text{ for } i=1 , \ldots, n-1\rangle.
\end{gather*}

Both algebras are quasihereditary with respect to the natural order on the set $\underline{n}=\{1,\ldots , n\}$. The algebra $A$ is the Auslander algebra of $K[x]/\langle x^n \rangle$ and $A'$ is the dual extension algebra of the linearly oriented quiver of type $\mathbb{A}_n$, in the sense of \cite[§$1.3$]{MR1285702}. 

The projective indecomposable $A$-module with top $L_1$ is given by
\begin{equation}
\label{eq:projectivep1}
\begin{tikzcd}[column sep=tiny, row sep=tiny]
\phantom{...}1\phantom{...} \arrow[d, no head]                             &    \phantom{n-1}                                              &                   \phantom{n-1}                                 &                         \phantom{n-1}                         &   \phantom{n-1}\\
2 \arrow[d, no head] \arrow[rd, no head]         &                                                  &                                                    &                                                  &   \\
3 \arrow[d, no head] \arrow[rd, no head]         & 1  \arrow[d, Green,no head]                             &                                                    &                                                  &   \\
4 \arrow[d, no head, dotted] \arrow[rd, no head] & 2 \arrow[d, Green,no head] \arrow[rd, no head]         &                                                    &                                                  &   \\
\circled{$n$} \arrow[rd, no head]                            & 3 \arrow[rd, no head] \arrow[d, Green,no head, dotted] & 1 \arrow[d ,Green, no head]                               &                                                  &   \\
                                                 & \circled{$n-1$} \arrow[rd, no head]                          & 2                                                  &                                                  &   \\
                                                 &                                                  & \circled{$n-2$} \arrow[rd, no head] \arrow[u, Green,no head, dotted] & 1 \arrow[d, Green,no head, dotted] \arrow[lu, no head] &   \\
                                                 &                                                  &                                                    & \circled{$n-3$} \arrow[rd, no head, dotted]                  &   \\
                                                 &                                                  &                                                    &                                                  & \circled{$1$}
\end{tikzcd}.
\end{equation}
The remaining projective indecomposable $A$-modules are submodules of $P_1$. To be precise, $P_i$ can be identified with the largest submodule of $P_1$ with top $L_i$. The projective $A'$-module with simple top labelled by $1$ is similar to \eqref{eq:projectivep1}, but having the green edges removed. As before, the rest of the projective indecomposable $A'$-modules are submodules of the projective with top labelled by $1$. In both cases, the standard modules are uniserial and have the same composition series. The module $\Delta_n$ is marked by rectangles in \eqref{eq:projectivep1} and $\Delta_i\cong \RAD{n-i}{ \Delta_n}$ for every $i \in \underline{n}$. The multiplicities $[\Delta_i:L_j]$ and $[\nabla_i:L_j]=(P_j:\Delta_i)$ are the same for $A$ and $A'$ and the dimension of the $\HOM$-spaces between standard modules also coincides for both algebras. 

Suppose that $\mathfrak{B}$ and $\mathfrak{B}'$ are minimal regular directed bocses whose right algebras are equivalent to $(A,\underline{n},\leq)$ and to $(A',\underline{n},\leq)$, respectively. It follows from Theorem \ref{thm:mainthm2} that the bocses $\mathfrak{B}$ and $\mathfrak{B}'$ have a lot of data in common, but they cannot be isomorphic since $(A,\underline{n},\leq)$ and $(A',\underline{n},\leq)$ are not equivalent.

Suppose that $B$ is a regular exact Borel subalgebra of some quasihereditary algebra $(R,\underline{n},\leq)$ which is either equivalent to $(A,\underline{n},\leq)$ or to $(A',\underline{n},\leq)$ and let $\Rest:\MOD{R}\rightarrow\MOD{B}$ be the associated restriction functor. Using \eqref{eq:recursivevector}, we get $v_1={\epsilon}_1$, $v_2={\epsilon}_2$ and
\begin{align*}
v_i&={\epsilon}_i+\sum_{\substack{1\leq k\leq i-2 \\ k\leq j\leq i-1}} (P_k:\Delta_j)\dim\big(\Hom{A}{\Delta_j}{\Delta_i}\big)v_k-\sum_{1\leq j\leq i-2}[\Delta_i:L_j]v_j \\
&={\epsilon}_i+\sum_{1\leq k\leq i-2} (i-k)v_k-\sum_{1\leq j\leq i-2} v_j={\epsilon}_i+\sum_{1\leq k\leq i-2} (i-k-1)v_k,
\end{align*}
for $2<i\leq n$. It then follows that
\[v_i=\sum_{k=1}^{i-2}2^{i-k-2}{\epsilon}_k+{\epsilon}_i.\]
According to Theorem \ref{thm:algo1} and Corollary \ref{cor:algo1}, we have
\[
[\Rest(L_i^R):L_j^B]=v_{ij}=\begin{cases} 2^{i-j-2} &\mbox{if } j\leq i-2 \\ 
1 & \mbox{if } j=i \\
0& \mbox{otherwise} \end{cases}, \quad \ell(\Rest(L_i^R))=l_i=\sum_{j=0}^{i-3}2^j+1.
\]
So $l_i=1$ if $i=1$ and $l_i=2^{i-2}$ if $1<i\leq n$, and the minimal good representative of $[(A,\underline{n},\leq)]$ (resp.~of $[(A',\underline{n},\leq)]$) is isomorphic to the opposite of the endomorphism algebra of $P_1\oplus (\bigoplus_{i=2}^n {P_i}^{2^{i-2}})$ over $A$ (resp.~over $A'$). 

By Theorem \ref{thm:mainthm2}, the composition factors of the projective indecomposable modules over the regular exact Borel subalgebra $B$ of $(R,\underline{n},\leq)$ are given by
\[
[P_j^B:L_i^B]=\sum_{j\leq k\leq i}(P_k:\Delta_i) v_{kj}=\sum_{j\leq k\leq i} v_{kj}=1 + \sum_{k=0}^{i-j-2} 2^k=\begin{cases} 2^{i-j-1} &\mbox{if }i>j \\ 
1 & \mbox{if } i=j \\
0& \mbox{otherwise} \end{cases}.
\]
Assuming additionally that $B$ is basic, we obtain
\[
\dim P_j^B=\ell\left(P_j^B \right) =\sum_{i=j+1}^n 2^{i-j-1}+1=\sum_{i=0}^{n-j-1} 2^{i}+1=2^{n-j}
\]
and $\dim B=\sum_{j=1}^n  2^{n-j}=2^n-1$.
\end{ex}
\section{Basic quasihereditary algebras with a regular exact Borel subalgebra}
\label{sec:last}
The results in Section \ref{sec:goodrepresentative} shall now be applied to characterise the basic quasihereditary algebras that admit a regular exact Borel subalgebra. Before doing so, we recall the notion of right minimal approximation and state an alternative description of the category of $\nabla$-filtered modules over a quasihereditary algebra. 

A map $f:X \rightarrow Y$ in $\Mod{A}$ is called a \emph{right minimal morphism}\index{right minimal morphism} if every endomorphism $g: X \rightarrow X$ satisfying $f=f \circ g$ is an automorphism.  See \cite[Section 1]{prepro} or \cite[Proposition $1.1$]{ausrei} for the key properties of minimal morphisms. Note, in particular, that nonzero morphisms with indecomposable domain are always right minimal. 

Let now $\mathcal{X}$ be a class of modules in $\Mod{A}$. A morphism $f: X \rightarrow Y$ in $\Mod{A}$, with $X$ in $\mathcal{X}$, is said to be a \emph{right $\mathcal{X}$-approximation} of $Y$ if the map 
\[\Hom{A}{X'}{f}:\Hom{A}{X'}{X}\longrightarrow  \Hom{A}{X'}{Y}\]
is surjective for all $X'$ in $\mathcal{X}$. Finally, say that a map is a \emph{right minimal $\mathcal{X}$-approximation} if it is both a right $\mathcal{X}$-approximation and a right minimal morphism.

In order to prove the next result, the following description of the category of $\nabla$-filtered modules over a quasihereditary algebra $(A,\Phi,\unlhd)$ is necessary:
\begin{align}
\label{eq:fdelta1}\mathcal{F}(\nabla)&=\{Y \in \Mod{A}\mid \Ext{A}{1}{\Delta_i}{Y}=0,\,\forall i \in \Phi\}\\
\label{eq:fdelta2}&=\{Y \in \Mod{A}\mid \Ext{A}{n}{X}{Y}=0,\, \forall X\in\mathcal{F}(\Delta), \forall n\geq 1\}.
\end{align}
For a proof of the identities above, we refer to \cite[Theorem 1]{MR1211481}.

\begin{prop}
\label{prop:almostlast}
Let $(A,\Phi,\unlhd)$ be a quasihereditary algebra and consider the corresponding families of integers $(v_{ij})_{i,j \in \Phi}$ and $(l_i)_{i\in \Phi}$ described, respectively, in Theorem \ref{thm:algo1} and in Corollary \ref{cor:algo1}. Fix $i\in \Phi$. The following conditions are equivalent:
\begin{enumerate}
\item $l_i=1$;
\item $v_{ij}=\delta_{ij}$ for every $j\in \Phi$;
\item $\Delta_i$ is a right  $\mathcal{F}(\Delta)$-approximation of the simple $A$-module $L_i$;
\item $\Rad{\Delta_i}$ belongs to $\mathcal{F}(\nabla)$;
\item the map $\Ext{A}{1}{X}{\pi_i}:\Ext{A}{1}{X}{\Delta_i}\rightarrow\Ext{A}{1}{X}{L_i}$, where $\pi_i$ denotes the epic from $\Delta_i$ to $L_i$, is an isomorphism for every $X$ in $\mathcal{F}(\Delta)$;
\item the map $\Ext{A}{1}{\Delta_j}{\pi_i}:\Ext{A}{1}{\Delta_j}{\Delta_i}\rightarrow\Ext{A}{1}{\Delta_j}{L_i}$, where $\pi_i$ denotes the epic from $\Delta_i$ to $L_i$, is an isomorphism for every $j\lhd i$. 
\end{enumerate}
\end{prop}
\begin{proof}
$(1)\Rightarrow (2)$: Combine Theorem \ref{thm:algo1} and \eqref{eq:formulali} in Corollary \ref{cor:algo1}.
 
$(2)\Rightarrow(3)$: Suppose that $v_{ij}=\delta_{ij}$ for every $j\in \Phi$. Let $\mathfrak{B}=(B,W,\mu, \varepsilon)$ be some regular directed bocs whose right algebra is equivalent to $(A,\Phi,\unlhd)$: by Thereom \ref{thm:kkorephrased} such a bocs always exist. Theorem \ref{thm:algo1} implies that $\Rest(L_i^{R_{\mathfrak{B}}})\cong L_i^B$. Consequently $R_{\mathfrak{B}}\otimes_B\Rest(L_i^{R_{\mathfrak{B}}})\cong \Delta^{R_{\mathfrak{B}}}_i$, since $B$ is an exact Borel subalgebra of $R_{\mathfrak{B}}$. Denote by $\theta$ the counit of the adjunction $R_{\mathfrak{B}}\otimes_B - \dashv \Rest$. According to Theorems $10.4$ and $11.3$ in \cite{MR3228437}, the category of $\Delta$-filtered $R_{\mathfrak{B}}$-modules coincides with the full subcategory of $\Mod{R_{\mathfrak{B}}}$ of modules isomorphic to $R_{\mathfrak{B}}\otimes_B X$ for $X$ in $\Mod{B}$. The universal property of $\theta$ then guarantees that $\theta_X:R_{\mathfrak{B}}\otimes_B\Rest(X) \rightarrow X$ is a right $\mathcal{F}(\Delta)$-approximation of $X$ for every $X$ in $\Mod{R_{\mathfrak{B}}}$. In particular, $\Delta_i^{R_{\mathfrak{B}}}$ is a right $\mathcal{F}(\Delta)$-approximation of $L_i^{R_{\mathfrak{B}}}$. Since $(A,\Phi,\unlhd)$ is equivalent to $R_{\mathfrak{B}}$, we deduce that $\Delta_i$ is a right $\mathcal{F}(\Delta)$-approximation of the simple $A$-module $L_i$.

$(3)\Rightarrow (4)$: Assume that (3) holds. The epic $\pi_i:\Delta_i \twoheadrightarrow L_i$ is then a right $\mathcal{F}(\Delta)$-approximation. Actually, the map $\pi_i$ is a right minimal $\mathcal{F}(\Delta)$-approximation since the module $\Delta_i$ is indecomposable. Using Wakamatsu's Lemma (\cite[Lemma 1.3]{ausrei}) and the identity \eqref{eq:fdelta1}, we conclude that $\Ker{\pi_i}=\Rad{\Delta_i}$ lies in $\mathcal{F}(\nabla)$.

$(4)\Rightarrow (1)$: Suppose that $\Rad{\Delta_i}\in \mathcal{F}(\nabla)$. We show that $l_i=1$. Recall the formula \eqref{eq:multiplicities} and note that
\begin{align*}
l_i&=1+ \sum_{\substack{j,k \in \Phi \\ k\unlhd j \lhd i}}  l_k  [\nabla_j: L_k] \dim\left( \Hom{A}{\Delta_j}{\Delta_i}\right)  - \sum_{\substack{j\in \Phi \\  j \lhd i}} l_j  [\Delta_i :L_j] \\
&=  1+\sum_{\substack{j,k \in \Phi \\ k\unlhd j \lhd i}}  l_k  [\nabla_j: L_k] \dim\left( \Hom{A}{\Delta_j}{\Rad{\Delta_i}}\right)  - \sum_{\substack{j\in \Phi \\  j \lhd i}} l_j  [\Rad{\Delta_i} :L_j]\\
&=  1+\sum_{\substack{j,k \in \Phi \\ k\unlhd j \lhd i}}  l_k  [\nabla_j: L_k] (\Rad{\Delta_i}:\nabla_j)  - \sum_{\substack{j\in \Phi \\  j \lhd i}} l_j  [\Rad{\Delta_i} :L_j],
\end{align*}
where the last equality follows from Remark \ref{rem:bgg}. Observe now that
\[\sum_{\substack{j \in \Phi \\ k\unlhd j \lhd i}}  [\nabla_j: L_k] (\Rad{\Delta_i}:\nabla_j) =[\Rad{\Delta_i}:L_k].\]
This implies that $l_i=1$.

$(4)\Rightarrow (5)$: Suppose that $\Rad{\Delta_i}$ is $\nabla$-filtered. Let $X$ be in $\mathcal{F}(\Delta)$, apply the functor $\Hom{A}{X}{-}$ to
\begin{equation}
\label{eq:ses}
\begin{tikzcd}[ampersand replacement=\&]
0 \arrow{r} \& \Rad{\Delta_{i}} \arrow{r} \& \Delta_{i} \arrow{r}{\pi_i} \& L_{i} \arrow{r} \& 0
\end{tikzcd}
\end{equation}
and consider the corresponding long exact sequence. Using \eqref{eq:fdelta2}, we get an isomorphism $\Ext{A}{1}{X}{\pi_i}:\Ext{A}{1}{X}{\Delta_i}\rightarrow\Ext{A}{1}{X}{L_i}$.

$(5)\Rightarrow (6)$: This implication is clear.

$(6)\Rightarrow (4)$: Let $j\in\Phi$, apply the functor $\Hom{A}{\Delta_j}{-}$ to the short exact \eqref{eq:ses} and consider the corresponding long exact sequence
\[
\begin{tikzcd}[column sep=normal]
0 \arrow[r]&  \Hom{A}{\Delta_j}{\Rad{\Delta_i}} \arrow[r]
    & \Hom{A}{\Delta_j}{\Delta_i} \arrow[r,"f"]
        \arrow[d, phantom, ""{coordinate, name=Z}]
      & \Hom{A}{\Delta_j}{L_i} \arrow[dll,"d"',rounded corners,to path={ -- ([xshift=2ex]\tikztostart.east)
|- (Z) [near end]\tikztonodes
-| ([xshift=-2ex]\tikztotarget.west) -- (\tikztotarget)}] \\
& \Ext{A}{1}{\Delta_j}{\Rad{\Delta_i}} \arrow[r,"g"]
    & \Ext{A}{1}{\Delta_j}{\Delta_i}\arrow[r,"h"]
&\Ext{A}{1}{\Delta_j}{L_i} \,.\end{tikzcd}
\]
If $j\lhd i$, then $h$ is an isomorphism by hypothesis and $\Hom{A}{\Delta_j}{L_i}=0$, therefore $\Ext{A}{1}{\Delta_j}{\Rad{\Delta_i}} =0$. If $j \not\vartriangleleft i$, then $\Ext{A}{1}{\Delta_j}{\Delta_i}=0$ by condition (3) in the definition of quasihereditary algebra (Definition \ref{defi:qh}). Since $\Hom{A}{\Delta_j}{L_i}=0$ for $j\neq i$, it follows that $\Ext{A}{1}{\Delta_j}{\Rad{\Delta_i}} =0$ whenever $j \not\vartriangleleft i$ and $j\neq i$. Finally, if $j=i$, the map $f$ is an isomorphism and $\Ext{A}{1}{\Delta_i}{\Delta_i}=0$, so $\Ext{A}{1}{\Delta_i}{\Rad{\Delta_i}} =0$. This shows that $\Ext{A}{1}{\Delta_j}{\Rad{\Delta_i}} =0$ for every $j\in \Phi$. According to \eqref{eq:fdelta1}, the module $\Rad{\Delta_i}$ lies in $\mathcal{F}(\nabla)$.
\end{proof}

As a consequence of Proposition \ref{prop:almostlast}, we determine when all quasihereditary algebras in an equivalence class are good and we also obtain a characterisation of the basic quasihereditary algebras that possess a regular exact Borel subalgebra.

\begin{thm}
\label{thm:basicqh}
Let $(A,\Phi,\unlhd)$ be a quasihereditary algebra. Consider the corresponding matrix $V_{[(A,\Phi,\unlhd)]}=(v_{ij})_{i,j \in \Phi}$ and the sequence $(l_i)_{i\in \Phi}$ described, respectively, in Theorem \ref{thm:algo1} and Remark \ref{rem:matrix}, and in Corollary \ref{cor:algo1}. The following conditions are equivalent:
\begin{enumerate}
\item every algebra in $[(A,\Phi,\unlhd)]$ is good (i.e.~every quasihereditary algebra equivalent to $(A,\Phi,\unlhd)$ has a regular exact Borel subalgebra);
	\item the minimal good representative of $[(A,\Phi,\unlhd)]$ is basic;
	\item the sequence $(l_i)_{i\in\Phi}$ is constant and equal to one;
\item $V_{[(A,\Phi,\lhd)]}$ is the identity matrix;
\item for every $i \in \Phi$, $\Delta_i$ is a right $\mathcal{F}(\Delta)$-approximation of the simple $A$-module $L_i$;
\item $\Rad{\Delta_i}$ belongs to $\mathcal{F}(\nabla)$ for every $i \in \Phi$;
\item the map $\Ext{A}{1}{X}{\pi_i}:\Ext{A}{1}{X}{\Delta_i}\rightarrow\Ext{A}{1}{X}{L_i}$, where $\pi_i$ denotes the epic from $\Delta_i$ to $L_i$, is an isomorphism for every $X$ in $\mathcal{F}(\Delta)$ and every $i\in\Phi$;
\item the map $\Ext{A}{1}{\Delta_j}{\pi_i}:\Ext{A}{1}{\Delta_j}{\Delta_i}\rightarrow\Ext{A}{1}{\Delta_j}{L_i}$, where $\pi_i$ denotes the epic from $\Delta_i$ to $L_i$, is an isomorphism for every $i,j\in\Phi$ satisfying $j\lhd i$.
\end{enumerate}
Assuming that $A$ is basic, then the algebra $(A,\Phi,\unlhd)$ contains a regular exact Borel subalgebra if and only if it contains a basic regular exact Borel subalgebra if and only if one of the equivalent conditions (1) to (8) holds.
\end{thm}
\begin{proof}
By Corollary \ref{cor:mainprimitive}, the minimal good representative of $[(A,\Phi,\unlhd)]$ is basic if and only if $l_i=1$ for every $i\in\Phi$. The equivalence of the conditions (2) to (8) follows then from Proposition \ref{prop:almostlast}. We only show that assertion (1) is also equivalent to conditions (2) to (8) by the end of this proof.

For now assume, within this paragraph, that $A$ is basic.
We prove first that $(A,\Phi,\unlhd)$ has a regular exact Borel subalgebra if and only if it contains a basic regular exact Borel subalgebra. One of the implications is obvious. So suppose that $(A,\Phi,\unlhd)$ contains a regular exact Borel subalgebra $B$. We shall see that $B$ must be basic. In fact, by parts (1) and (2) of Theorem \ref{thm:mainsuper}, we have
\[	
1= \sum_{j \in \Phi}v_{ij}\dim L_j^B
\]
for every $i\in\Phi$. Since $v_{ij}\in \Znn$ and $v_{ii}=1$ (recall Theorem \ref{thm:algo1}), we must have $\dim L_i^B\leq 1$ for every $i\in\Phi$. Therefore $\dim L_i^B=1$ for all $i\in\Phi$ and $B$ is basic, as claimed. Recall now part (5) of Theorem \ref{thm:mainsuper} and note that $(A,\Phi,\unlhd)$ contains a basic regular exact Borel subalgebra if and only if $l_i=1$ for every $i\in\Phi$, that is, if and only if condition (3) in the statement of the theorem holds, i.e.~if and only if the equivalent conditions (2) to (8) hold for $(A,\Phi,\unlhd)$.

It remains to show that assertion (1) is equivalent to conditions (2) to (8) in the general setting where $(A,\Phi,\unlhd)$ is not necessarily basic. So suppose again that $(A,\Phi,\unlhd)$ is an arbitrary quasihereditary algebra. By part (1) of Theorem \ref{thm:mainsuper}, assertion (4) implies (1). Suppose now that condition (1) holds. It follows that the basic version of $(A,\Phi,\unlhd)$, say $(A',\Phi,\unlhd)$, has a regular exact Borel subalgebra. As seen in the previous paragraph, this implies that conditions (2) to (8) hold for $(A',\Phi,\unlhd)$. Since $[(A',\Phi,\unlhd)]=[(A,\Phi,\unlhd)]$, assertion (2) must hold $(A,\Phi,\unlhd)$. ~This concludes the proof of the theorem.
\end{proof}

\begin{ex}
Let $(A,\Phi,\unlhd)$ be a quasihereditary algebra such that all standard modules are simple. Then $\Rad{\Delta_i}=0$ for every $i \in \Phi$, so the radical of every standard module is $\nabla$-filtered. By Theorem \ref{thm:basicqh}, every quasihereditary algebra equivalent to $(A,\Phi,\unlhd)$ has a regular exact Borel subalgebra. In particular, $(A,\Phi,\unlhd)$ has a regular exact Borel subalgebra. Using the definition, it is easy to conclude that $A$ is its own regular exact Borel subalgebra.
\end{ex}

\begin{ex}
	Let $(A,\Phi,\unlhd)$ be a quasihereditary algebra such that all costandard modules are simple. Then $\Rad{\Delta_i}$ is trivially $\nabla$-filtered for every $i \in \Phi$ and Theorem \ref{thm:basicqh} assures that $(A,\Phi,\unlhd)$ (as well as every algebra in $[(A,\Phi,\unlhd)]$) has a regular exact Borel subalgebra. We claim that the semisimple algebra $A/\Rad{A}$ is a good exact Borel subalgebra of $(A,\Phi,\unlhd)$. In order to see this, observe first that $l_i=1$ for every $i \in \Phi$, by part (3) of Theorem \ref{thm:basicqh}. Part (2) of Theorem \ref{thm:mainthm2} subsequently implies that the injective indecomposable modules over a regular exact Borel subalgebra $B$ of $(A,\Phi,\unlhd)$ have to be simple, i.e.~$B$ must be isomorphic to $A/\Rad{A}$. The Wedderburn--Malcev Theorem (see Theorem $3.6.9$ in \cite{Zimmermann:1952386}) also confirms that $A/\Rad{A}$ is a subalgebra of $(A,\Phi,\unlhd)$ which is additionally a regular exact Borel subalgebra.
\end{ex}

\subsection{An application}

We conclude this paper by describing a class of quasihereditary algebras whose Ringel dual contains a basic regular exact Borel subalgebra. 

The \emph{Ringel dual} $\mathcal{R}(A)$ of a quasihereditary algebra $(A,\Phi, \unlhd)$ was first considered in \cite{last} and it is defined as $\End{A}{T}\op$, where the module $T$ is a multiplicity-free additive generator of the category $\mathcal{F}(\Delta) \cap \mathcal{F}(\nabla)$. It is common to refer to a module in $\mathcal{F}(\Delta) \cap \mathcal{F}(\nabla)$ as a \emph{tilting module}. The indecomposable tilting modules are in bijection with the elements of $\Phi$. We write $T=\bigoplus_{i \in \Phi} T_i$, where the module $T_i$ in $\mathcal{F}(\Delta) \cap \mathcal{F}(\nabla)$ is characterised by the following property: it is filtered by stardard modules of the form $\Delta_j$ with $j\unlhd i$ and satisfies $(T_i:\Delta_i)=1$. The simple modules over $\mathcal{R}(A)$ are therefore naturally indexed by $\Phi$ and the Ringel dual $\mathcal{R}(A)$ of $(A,\Phi, \unlhd)$ is actually a quasihereditary algebra with respect to the opposite poset $(\Phi, \unlhd\op)$ of $(\Phi, \unlhd)$. In fact, the category of $\Delta$-filtered modules over $(\mathcal{R}(A),\Phi, \unlhd\op)$ is equivalent to the category of $\nabla$-filtered modules over $(A,\Phi, \unlhd)$. According to the definition given in here, the Ringel dual is always a basic quasihereditary algebra.

We will show that the Ringel dual of the dual extension of the incidence algebra of a tree always has a basic good exact Borel subalgebra. First, some notions and notation must be introduced. Recall that a \emph{tree} $(\Phi, \unlhd)$ is a finite poset such that the subposet $\{j\in \Phi\mid j \unlhd i\}$ is linearly ordered for every $i\in \Phi$. Given a tree $(\Phi, \unlhd)$, consider an associated quiver $Q_{(\Phi, \unlhd)}$ constructed in the following way:
\begin{enumerate}
\item the vertices of $Q_{(\Phi, \unlhd)}$ are labelled by $\Phi$;
\item for every pair $j\lhd i$ with $j$ an immediate predecessor of $i$, draw one arrow $\alpha_{ji}$ from $j$ to $i$ and one arrow $\beta_{ij}$ from $i$ to $j$.
\end{enumerate}
Define $A_{(\Phi, \unlhd)}$ as the bound quiver algebra of $Q_{(\Phi, \unlhd)}$ with relations $\alpha_{jk}\beta_{ij}=0$, $i,j,k\in \Phi$. Let $i^-$ be the (unique) immediate predecessor of some $i$ in $(\Phi,\unlhd)$ and let $i_1^+, \ldots,i_m^+$ be the list of all elements in $\Phi$ having $i$ as an immediate predecessor. Locally, around the vertex $i$, the quiver $Q_{(\Phi, \unlhd)}$ may be depicted as
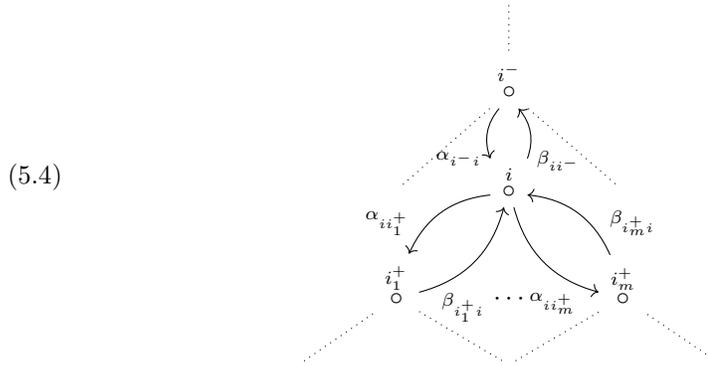
\begin{equation}
\label{eq:lastone}
\begin{tikzcd}[column sep= normal]
   &                                                                                                           & {} \arrow[d, no head, dotted]                                                                                                        &                                                                                                           &    \\
   &                                                                                                           & \overset{i^-}{\circ} \arrow[d,phantom, symbol=] \arrow[ld, no head, dotted] \arrow[rd, no head, dotted] \arrow[d, "\alpha_{i^{-}i}"', near end,bend right]                                &                                                                                                           &    \\
   & {}                                                                                                        & \overset{i}{\circ} \arrow[dr,phantom, symbol=, near start] \arrow[dl,phantom, symbol=, near start] \arrow[ld, "\alpha_{ii^{+}_1}"', very near end, bend right] \arrow[rd, "\alpha_{ii^{+}_m}"',very near end, bend right] \arrow[u, "\beta_{ii^{-}}"', near start, bend right] & {}                                                                                                        &    \\
   & \overset{i_1^+}{\circ} \arrow[rd, no head, dotted] \arrow[ld, no head, dotted] \arrow[ru, "\beta_{i_1^{+}i}"', very near start,bend right] & \cdots                                                                                                                               & \overset{i_m^+}{\circ} \arrow[ld, no head, dotted] \arrow[rd, no head, dotted] \arrow[lu, "\beta_{i_m^{+}i}"', very near start, bend right] &    \\
{} &                                                                                                           & {}                                                                                                                                   &                                                                                                           & {}
\end{tikzcd}.
\end{equation}

If $(\Phi, \unlhd)$ is a tree, then the \emph{dual extension of the incidence algebra} of the poset $(\Phi, \unlhd)$ coincides with the algebra $A_{(\Phi, \unlhd)}$. We refer to \cite[§$1.6$]{xi1994quasi}, \cite[§$1.3$]{MR1285702} and \cite[§§$1.2$--$1.5$]{MR1337121} for the general definition of dual extension algebra and of incidence algebra of a poset and point out that the incidence algebra of a tree is a hereditary algebra. According to \cite[§$1.6$]{xi1994quasi} (see also \cite[§$1.7$]{MR1337121}), the algebra $A_{(\Phi, \unlhd)}$ is quasihereditary with respect to $(\Phi, \unlhd)$ and its Ringel dual has been studied in \cite{MR1690433}. We prove the following result concerning the Ringel dual of $(A_{(\Phi, \unlhd)},\Phi, \unlhd)$ for a tree $(\Phi, \unlhd)$.

\begin{prop}
\label{prop:last}
The Ringel dual of the dual extension of the incidence algebra of a tree has a basic regular exact Borel subalgebra. 
\end{prop}
\begin{rem}
Using Theorem \ref{thm:basicqh}, it is possible to rephrase the statement of Proposition \ref{prop:last} in multiple equivalent ways. In particular, it shall follow from Proposition \ref{prop:last} that the algebra $(\mathcal{R}(A_{(\Phi, \unlhd)}),\Phi, \unlhd\op)$ satisfies all the conditions in the statement of Theorem \ref{thm:basicqh} for any choice of a tree $(\Phi, \unlhd)$.
\end{rem}
Lemma \ref{lem:last}, together with Theorem \ref{thm:basicqh} and \cite[Lemma $2.1$]{MR1396858}, will be the key ingredients in the proof of Proposition \ref{prop:last}.
\begin{lem}
	\label{lem:last}
	Let $(A,\Phi, \unlhd)$ be a quasihereditary algebra and let $m\in \Phi$. Suppose that there exist short exact sequences
	\begin{equation}
	\label{eq:ses2}
	\begin{tikzcd}[ampersand replacement=\&]
	0 \arrow{r} \& \Delta_k \arrow{r} \& T_k \arrow{r} \& \bigoplus\limits_{\substack{l \in \Phi \\ l\lhd  k}} T_l \arrow{r} \& 0
	\end{tikzcd}
	\end{equation}
	for every $k\in \Phi$ with $k\unlhd m$. The following identity holds for every $i,k\in \Phi$ with $i \lhd k\unlhd m$: 
	\begin{equation}
	\label{eq:reallylast}
	(T_k :\Delta_i)=\sum_{\substack{j \in \Phi \\ i\lhd j\unlhd k}} (T_k: \Delta_j).
	\end{equation}
\end{lem}
\begin{proof}
	We show \eqref{eq:reallylast} by induction on $m\in \Phi$. Suppose first that $m$ is an element in $(\Phi, \unlhd)$ whose immediate predecessors are all minimal in $(\Phi, \unlhd)$. If $m^-$ is an immediate predecessor of $m$, then $(T_m :\Delta_{m^-})=1$ by \eqref{eq:ses2} and
	\[\sum_{\substack{j \in \Phi \\ m^- \lhd j\unlhd m}} (T_m: \Delta_j)=(T_m:\Delta_m)=1.
	\]
	Hence the base case of the induction holds. 
	
	For $m\in \Phi$ arbitrary consider $i,k \in \Phi$ with $i \lhd k\unlhd m$. Using \eqref{eq:ses2}, we obtain
	\begin{align*}
	\sum_{\substack{j \in \Phi \\ i\lhd j\unlhd k}} (T_k: \Delta_j)&= \sum_{\substack{j \in \Phi \\ i\lhd j\unlhd k}} \left( (\Delta_k: \Delta_j)+ \sum_{\substack{l \in \Phi \\ l\lhd k}} (T_l:\Delta_j)\right) \\
	&=(\Delta_k: \Delta_k)+ \sum_{\substack{l \in \Phi \\ l\lhd k}} (T_l:\Delta_k) + \sum_{\substack{j \in \Phi \\ i\lhd j\lhd k}} \sum_{\substack{l \in \Phi \\ l\lhd k}} (T_l:\Delta_j) \\
	&=1+ \sum_{\substack{l \in \Phi \\ l\lhd k}}\sum_{\substack{j \in \Phi \\ i\lhd j\lhd k}}  (T_l:\Delta_j) =1+ \sum_{\substack{l \in \Phi \\ l\lhd k}}\sum_{\substack{j \in \Phi \\ i\lhd j\unlhd l}}  (T_l:\Delta_j) \\
	&=1+ \sum_{\substack{l \in \Phi \\  i\lhd l\lhd k}}\sum_{\substack{j \in \Phi \\ i\lhd j\unlhd l}}  (T_l:\Delta_j) =1+ \sum_{\substack{l \in \Phi \\  i\lhd l\lhd k}}(T_l:\Delta_i) ,
	\end{align*}
	where the last equality follows from the induction hypothesis. Using \eqref{eq:ses2}, we get
	\[(T_k:\Delta_i)=(\Delta_k:\Delta_i)+\sum_{\substack{l \in \Phi \\ l\lhd k}} (T_l:\Delta_i)=\sum_{\substack{l \in \Phi \\ i\unlhd l \lhd k}} (T_l:\Delta_i)=1+ \sum_{\substack{l \in \Phi \\ i\lhd l \lhd k}} (T_l:\Delta_i).\]
	This completes the proof.
\end{proof}

\begin{proof}[Proof of Proposition \ref{prop:last}]
Let $(\Phi, \unlhd)$ be a tree. Observe that the quasihereditary algebra $(\mathcal{R}(A_{(\Phi, \unlhd)}),\Phi, \unlhd\op)$ is basic by the definition of Ringel dual. By Theorem \ref{thm:basicqh}, it is enough to prove that the sequence of integers $(l_i)_{i\in\Phi}$ associated to $(\mathcal{R}(A_{(\Phi, \unlhd)}),\Phi, \unlhd\op)$ is constant and equal to one. Use the symbol $'$ to denote the modules over $\mathcal{R}(A_{(\Phi, \unlhd)})$. By applying \eqref{eq:multiplicities} and the Bernstein--Gelfand--Gelfand Reciprocity Law (see Remark \ref{rem:bgg}), we get
\begin{align*}
l_i&=1+ \sum_{\substack{j,k \in \Phi \\ k\unlhd\op j \lhd\op i}}  l_k  (P_k': \Delta_j') \dim\left( \Hom{\mathcal{R}(A_{(\Phi, \unlhd)})}{\Delta_j'}{\Delta_i'}\right)  - \sum_{\substack{k\in \Phi \\  k \lhd\op i}} l_k  (Q_k' :\nabla_i')  \\
&=1+ \sum_{\substack{j,k \in \Phi \\ i\lhd j \unlhd k}}  l_k  (T_k: \nabla_j) \dim\left( \Hom{A_{(\Phi, \unlhd)}}{\nabla_j}{\nabla_i}\right)  - \sum_{\substack{k\in \Phi \\  i \lhd k}} l_k  (T_k :\Delta_i) ,
\end{align*}
where the last equality results from the properties of the Ringel duality (see \cite[§6]{last}). By looking at \eqref{eq:lastone}, one easily sees that the costandard $A_{(\Phi, \unlhd)}$-modules are uniserial. In fact, the composition factors of $\nabla_i$ are in bijection with $\{j\in\Phi\mid j\unlhd i\}$ and the vector space $\Hom{A_{(\Phi, \unlhd)}}{\nabla_j}{\nabla_i}$ is $1$-dimensional for every $i,j\in \Phi$ with $i \lhd j$. As a result,
\begin{align*}
l_i &=1+ \sum_{\substack{k \in \Phi \\ i\lhd k}}  l_k \left(  \sum_{\substack{j \in \Phi \\ i\lhd j\unlhd k}} (T_k: \nabla_j) \right)  - \sum_{\substack{k\in \Phi \\  i \lhd k}} l_k  (T_k :\Delta_i) \\
&= 1+ \sum_{\substack{k \in \Phi \\ i\lhd k}}  l_k \left(  \left( \sum_{\substack{j \in \Phi \\ i\lhd j\unlhd k}} (T_k: \Delta_j) \right)  -  (T_k :\Delta_i) \right)  ,
\end{align*}
where the last identity follows from the fact that $\Mod{A_{(\Phi, \unlhd)}}$ has a simple preserving duality (see \cite[§§$1.3$, $1.4$, $1.6$]{xi1994quasi}). By \cite[Lemma $2.1$]{MR1396858} and Lemma \ref{lem:last}, we obtain $l_i=1$ for every $i\in \Phi$.
\end{proof}

\begin{ex}
\label{ex:last}
Fix $n\in \N$ and let $(A',\underline{n},\leq)$ be the quasihereditary algebra with $n$ simple modules discussed in Example \ref{ex:third}. This is the dual extension algebra of the incidence algebra of the tree $(\underline{n},\leq)$, that is, $A'=A_{(\underline{n},\leq)}$. By Proposition \ref{prop:last}, the Ringel dual of $(A',\underline{n},\leq)$ has a basic regular exact Borel subalgebra.
\end{ex}

\bibliographystyle{amsplain}
\bibliography{QHAlg}

\providecommand{\bysame}{\leavevmode\hbox to3em{\hrulefill}\thinspace}
\providecommand{\MR}{\relax\ifhmode\unskip\space\fi MR }
\providecommand{\MRhref}[2]{%
  \href{http://www.ams.org/mathscinet-getitem?mr=#1}{#2}
}
\providecommand{\href}[2]{#2}
\begin{thebibliography}{10}

\bibitem{ausrei}
M.~Auslander and I.~Reiten, \emph{Applications of contravariantly finite
  subcategories}, Adv. Math. \textbf{86} (1991), no.~1, 111--152.

\bibitem{prepro}
M.~Auslander and S.~O. Smal{\o}, \emph{Preprojective modules over {A}rtin
  algebras}, J. Algebra \textbf{66} (1980), no.~1, 61--122.

\bibitem{benson_1991}
D.~J. Benson, \emph{Representations and {C}ohomology}, Cambridge Studies in
  Advanced Mathematics, vol.~1, Cambridge University Press, 1991.

\bibitem{MR3800074}
A.~Bodzenta and J.~K\"{u}lshammer, \emph{Ringel duality as an instance of
  {K}oszul duality}, J. Algebra \textbf{506} (2018), 129--187.

\bibitem{bourbaki1998algebra}
N.~Bourbaki, \emph{Algebra {I}: {C}hapters 1-3}, Actualit{\'e}s scientifiques
  et industrielles, Springer, 1998.

\bibitem{doi:10.1112/blms.12331}
T.~Brzeziński, S.~Koenig, and J.~Külshammer, \emph{From quasi-hereditary
  algebras with exact {B}orel subalgebras to directed bocses}, Bull. Lond.
  Math. Soc. \textbf{52} (2020), no.~2, 367--378.

\bibitem{brzezinski2003corings}
T.~Brzeziński and R~Wisbauer, \emph{Corings and comodules}, London
  Mathematical Society Lecture Note Series, vol. 309, Cambridge University
  Press, Cambridge, 2003.

\bibitem{assforbocses}
W.~L. Burt and M.~C.~R. Butler, \emph{Almost split sequences for bocses},
  Representations of finite-dimensional algebras ({T}sukuba, 1990), CMS Conf.
  Proc., vol.~11, Amer. Math. Soc., Providence, RI, 1991, pp.~89--21.

\bibitem{clineparshallscott}
E.~Cline, B.~Parshall, and L.~Scott, \emph{Finite-dimensional algebras and
  highest weight categories}, J. Reine Angew. Math. \textbf{391} (1988),
  85--99.

\bibitem{thesis}
T.~Conde, \emph{On certain strongly quasihereditary algebras}, Ph.D. thesis,
  University of Oxford, 2016.

\bibitem{manuela}
\bysame, \emph{All exact {B}orel subalgebras and all directed bocses are
  normal}, arXiv preprint arXiv:2010.04128 (2020).

\bibitem{coulembier2019classification}
K.~Coulembier, \emph{The classification of blocks in {BGG} category
  $\mathcal{O}$}, Math. Z. (2019), 1--17.

\bibitem{MR1285702}
B.~M. Deng and C.~C. Xi, \emph{Quasi-hereditary algebras which are dual
  extensions of algebras}, Comm. Algebra \textbf{22} (1994), no.~12,
  4717--4735.

\bibitem{MR1337121}
\bysame, \emph{Quasi-hereditary algebras which are twisted double incidence
  algebras of posets}, Beitr\"{a}ge Algebra Geom. \textbf{36} (1995), no.~1.

\bibitem{MR1396858}
\bysame, \emph{Ringel duals of quasi-hereditary algebras}, Comm. Algebra
  \textbf{24} (1996), no.~9, 2825--2838.

\bibitem{MR987824}
V.~Dlab and C.~M. Ringel, \emph{Quasi-hereditary algebras}, Illinois J. Math.
  \textbf{33} (1989), no.~2, 280--291.

\bibitem{MR1211481}
\bysame, \emph{The module theoretical approach to quasi-hereditary algebras},
  Representations of algebras and related topics ({K}yoto, 1990), London Math.
  Soc. Lecture Note Ser., vol. 168, Cambridge Univ. Press, Cambridge, 1992,
  pp.~200--224.

\bibitem{MR866778}
S.~Donkin, \emph{On {S}chur algebras and related algebras. {I}}, J. Algebra
  \textbf{104} (1986), no.~2, 310--328.

\bibitem{MR916172}
\bysame, \emph{On {S}chur algebras and related algebras. {II}}, J. Algebra
  \textbf{111} (1987), no.~2, 354--364.

\bibitem{MR1206201}
K.~Erdmann, \emph{Schur algebras of finite type}, Quart. J. Math. Oxford Ser.
  (2) \textbf{44} (1993), no.~173, 17--41.

\bibitem{flores2020combinatorics}
M.~Flores, Y.~Kimura, and B.~Rognerud, \emph{Combinatorics of quasi-hereditary
  structures}, arXiv preprint arXiv:2004.04726 (2020).

\bibitem{MR2995138}
A.~Henke and S.~Koenig, \emph{Schur algebras of {B}rauer algebras {I}}, Math.
  Z. \textbf{272} (2012), no.~3-4, 729--759.

\bibitem{MR3175171}
\bysame, \emph{Schur algebras of {B}rauer algebras, {II}}, Math. Z.
  \textbf{276} (2014), no.~3-4, 1077--1099.

\bibitem{humphreys2008representations}
J.~E. Humphreys, \emph{Representations of {S}emisimple {L}ie {A}lgebras in the
  {BGG} {C}ategory $\mathcal{O}$}, vol.~94, American Mathematical Soc., 2008.

\bibitem{humphreys2012introduction}
\bysame, \emph{Introduction to {L}ie algebras and representation theory},
  vol.~9, Springer Science \& Business Media, 2012.

\bibitem{klamt2011ainfinity}
A.~Klamt, \emph{A-infinity structures on the algebra of extensions of {V}erma
  modules in the parabolic category {O}}, Diploma thesis, University of Bonn,
  2011.

\bibitem{KLAMT2012323}
A.~Klamt and C.~Stroppel, \emph{On the $\operatorname{Ext}$ algebras of
  parabolic {V}erma modules and {$A_{\infty}$}-structures}, J. Pure Appl.
  Algebra \textbf{216} (2012), no.~2, 323 -- 336.

\bibitem{MR1362252}
S.~Koenig, \emph{Exact {B}orel subalgebras of quasi-hereditary algebras. {I}},
  Math. Z. \textbf{220} (1995), no.~3, 399--426, With an appendix by Leonard
  Scott.

\bibitem{MR3228437}
S.~Koenig, J.~K{\"u}lshammer, and S.~Ovsienko, \emph{Quasi-hereditary algebras,
  exact {B}orel subalgebras, {$A_\infty$}-categories and boxes}, Adv. Math.
  \textbf{262} (2014), 546--592.

\bibitem{kulshammer2016bocs}
J.~K\"{u}lshammer, \emph{In the bocs seat: quasi-hereditary algebras and
  representation type}, Representation theory -- current trends and
  perspectives, EMS Ser. Congr. Rep., Eur. Math. Soc., Z\"{u}rich, 2017,
  pp.~375--426.

\bibitem{lang2005algebra}
S.~Lang, \emph{Algebra}, Graduate Texts in Mathematics, Springer New York,
  2005.

\bibitem{Moosonee}
B.~Parshall and L.~Scott, \emph{Derived categories, quasi-hereditary algebras,
  and algebraic groups}, Proc. Ottawa--Moosonee Workshop, vol.~3,
  Carleton--Ottawa Math. Lecture Note Series, 1988, pp.~1--105.

\bibitem{MR1048073}
B.~Parshall and J.~P. Wang, \emph{Quantum linear groups}, Mem. Amer. Math. Soc.
  \textbf{89} (1991), no.~439, vi+157.

\bibitem{riehl2017category}
E.~Riehl, \emph{Category theory in context}, Courier Dover Publications, 2017.

\bibitem{last}
C.~M. Ringel, \emph{The category of modules with good filtrations over a
  quasi-hereditary algebra has almost split sequences}, Math. Z. \textbf{208}
  (1991), no.~1, 209--223.

\bibitem{MR933417}
L.~Scott, \emph{Simulating algebraic geometry with algebra. {I}. {T}he
  algebraic theory of derived categories}, The {A}rcata {C}onference on
  {R}epresentations of {F}inite {G}roups ({A}rcata, {C}alif., 1986), Proc.
  Sympos. Pure Math., vol.~47, Amer. Math. Soc., Providence, RI, 1987,
  pp.~271--281.

\bibitem{MR118757}
C.~E. Watts, \emph{Intrinsic characterizations of some additive functors},
  Proc. Amer. Math. Soc. \textbf{11} (1960), 5--8.

\bibitem{wiebel}
C.~A. Weibel, \emph{An introduction to homological algebra}, Cambridge studies
  in advanced mathematics, Cambridge University Press, Cambridge, 1994.

\bibitem{xi1994quasi}
C.~C. Xi, \emph{Quasi-hereditary algebras with a duality}, J. Reine Angew.
  Math. \textbf{449} (1994), 201--216.

\bibitem{MR1690433}
Y.~H. Zhang, \emph{Ringel duals of extension algebras of posets}, Chinese Sci.
  Bull. \textbf{44} (1999), no.~2, 129--132.

\bibitem{Zimmermann:1952386}
A.~Zimmermann, \emph{{Representation theory: a homological algebra point of
  view}}, Algebra and Applications, Springer, 2014.

\end{thebibliography}

\end{document}